\documentclass[11pt,letterpaper,reqno]{amsart}

\title[]{Controlled Lagrangians and Stabilization of Euler--Poincar\'e Equations with Symmetry Breaking Nonholonomic Constraints}
\author{Jorge S. Garcia}
\author{Tomoki Ohsawa}
\address{Department of Mathematical Sciences, The University of Texas at Dallas, 800 W Campbell Rd, Richardson, TX 75080-3021}
\email{jorge.garcia4@utdallas.edu,tomoki@utdallas.edu}
\date{\today}



\usepackage{graphicx,amssymb,mathrsfs,paralist}
\usepackage[caption=false]{subfig}
\usepackage[margin=1in, marginpar=.5in]{geometry}

\usepackage[numbers,sort&compress]{natbib}
\usepackage[colorlinks=false]{hyperref}

\usepackage[nameinlink]{cleveref}

\theoremstyle{plain}
\newtheorem{theorem}{Theorem}
\newtheorem*{theorem*}{Theorem}

\newtheorem{proposition}[theorem]{Proposition}
\newtheorem*{EC}{Energy--Casimir Theorem~(\citet{Ae1992})}
\newtheorem{assumption}{Assumption}
\Crefname{assumption}{Assumption}{Assumptions}

\theoremstyle{definition}

\newtheorem{example}[theorem]{Example}

\theoremstyle{remark}
\newtheorem{remark}[theorem]{Remark}

\def\od#1#2{\frac{d#1}{d#2}}
\def\pd#1#2{\frac{\partial #1}{\partial #2}}
\def\fd#1#2{\frac{\delta #1}{\delta #2}}

\def\tpd#1#2{\partial #1/\partial #2}
\def\tfd#1#2{\delta #1/\delta #2}

\def\dzero#1#2{\left.\od{}{#1} #2 \right|_{#1=0}}

\def\parentheses#1{{\left(#1\right)}}

\def\tr{\mathop{\mathrm{tr}}\nolimits}

\def\Span{\mathop{\mathrm{span}}\nolimits} 

\def\norm#1{{\left\|#1\right\|}}
\def\abs#1{{\left|#1\right|}}

\def\R{\mathbb{R}}

\def\defeq{\mathrel{\mathop:}=}
\def\eqdef{=\mathrel{\mathop:}}
\def\setdef#1#2{{\left\{ #1 \ |\ #2 \right\}}}
\def\ip#1#2{{\left\langle#1,#2\right\rangle}}

\def\diag{\operatorname{diag}}

\def\eps{\varepsilon}

\def\SO{\mathsf{SO}}
\def\SE{\mathsf{SE}}

\def\se{\mathfrak{se}}
\def\so{\mathfrak{so}}

\def\d{{\bf d}}

\newcommand\Ad{\operatorname{Ad}}
\newcommand\ad{\operatorname{ad}}

\makeatletter
\DeclareFontFamily{OMX}{MnSymbolE}{}
\DeclareSymbolFont{MnLargeSymbols}{OMX}{MnSymbolE}{m}{n}
\SetSymbolFont{MnLargeSymbols}{bold}{OMX}{MnSymbolE}{b}{n}
\DeclareFontShape{OMX}{MnSymbolE}{m}{n}{
    <-6>  MnSymbolE5
   <6-7>  MnSymbolE6
   <7-8>  MnSymbolE7
   <8-9>  MnSymbolE8
   <9-10> MnSymbolE9
  <10-12> MnSymbolE10
  <12->   MnSymbolE12
}{}
\DeclareFontShape{OMX}{MnSymbolE}{b}{n}{
    <-6>  MnSymbolE-Bold5
   <6-7>  MnSymbolE-Bold6
   <7-8>  MnSymbolE-Bold7
   <8-9>  MnSymbolE-Bold8
   <9-10> MnSymbolE-Bold9
  <10-12> MnSymbolE-Bold10
  <12->   MnSymbolE-Bold12
}{}

\let\llangle\@undefined
\let\rrangle\@undefined
\DeclareMathDelimiter{\llangle}{\mathopen}%
                     {MnLargeSymbols}{'164}{MnLargeSymbols}{'164}
\DeclareMathDelimiter{\rrangle}{\mathclose}%
                     {MnLargeSymbols}{'171}{MnLargeSymbols}{'171}
\makeatother

\def\bOmega{\boldsymbol{\Omega}}
\def\bPi{\boldsymbol{\Pi}}

\def\bgamma{\boldsymbol{\gamma}}
\def\bGamma{\boldsymbol{\Gamma}}
\def\bzero{\mathbf{0}}
\def\bx{\mathbf{x}}
\def\bY{\mathbf{Y}}
\def\bE#1{\mathbf{E}_{#1}}
\def\be#1{\mathbf{e}_{#1}}

\begin{document}

\footskip=.5in

\begin{abstract}
  We extend the method of Controlled Lagrangians to nonholonomic Euler--Poincar\'e equations with advected parameters, specifically to those mechanical systems on Lie groups whose symmetry is broken not only by a potential force but also by nonholonomic constraints. We introduce advected-parameter-dependent quasivelocities in order to systematically eliminate the Lagrange multipliers in the nonholonomic Euler--Poincar\'e equations. The quasivelocities facilitate the method of Controlled Lagrangians for these systems, and lead to matching conditions that are similar to those by Bloch, Leonard, and Marsden for the standard holonomic Euler--Poincar\'e equation. Our motivating example is what we call the pendulum skate, a simple model of a figure skater developed by Gzenda and Putkaradze. We show that the upright spinning of the pendulum skate is stable under certain conditions, whereas the upright sliding equilibrium is always unstable. Using the matching condition, we derive a control law to stabilize the sliding equilibrium.
\end{abstract}

\maketitle

\section{Introduction}
\subsection{Motivating Example: Pendulum Skate}
Consider what we call the \textit{pendulum skate} shown in \Cref{fig:PendulumSkate}.
It is a simple model for a figure skater developed and analyzed by \citet{GzPu2020}, and consists of a skate---sliding without friction on the surface---with a pendulum rigidly attached to it.
We note that, in this paper, we shall focus on the integrable case where the center of mass is above the point of contact.

\begin{figure}[hbtp]
  \centering
  \includegraphics[width=0.8\linewidth]{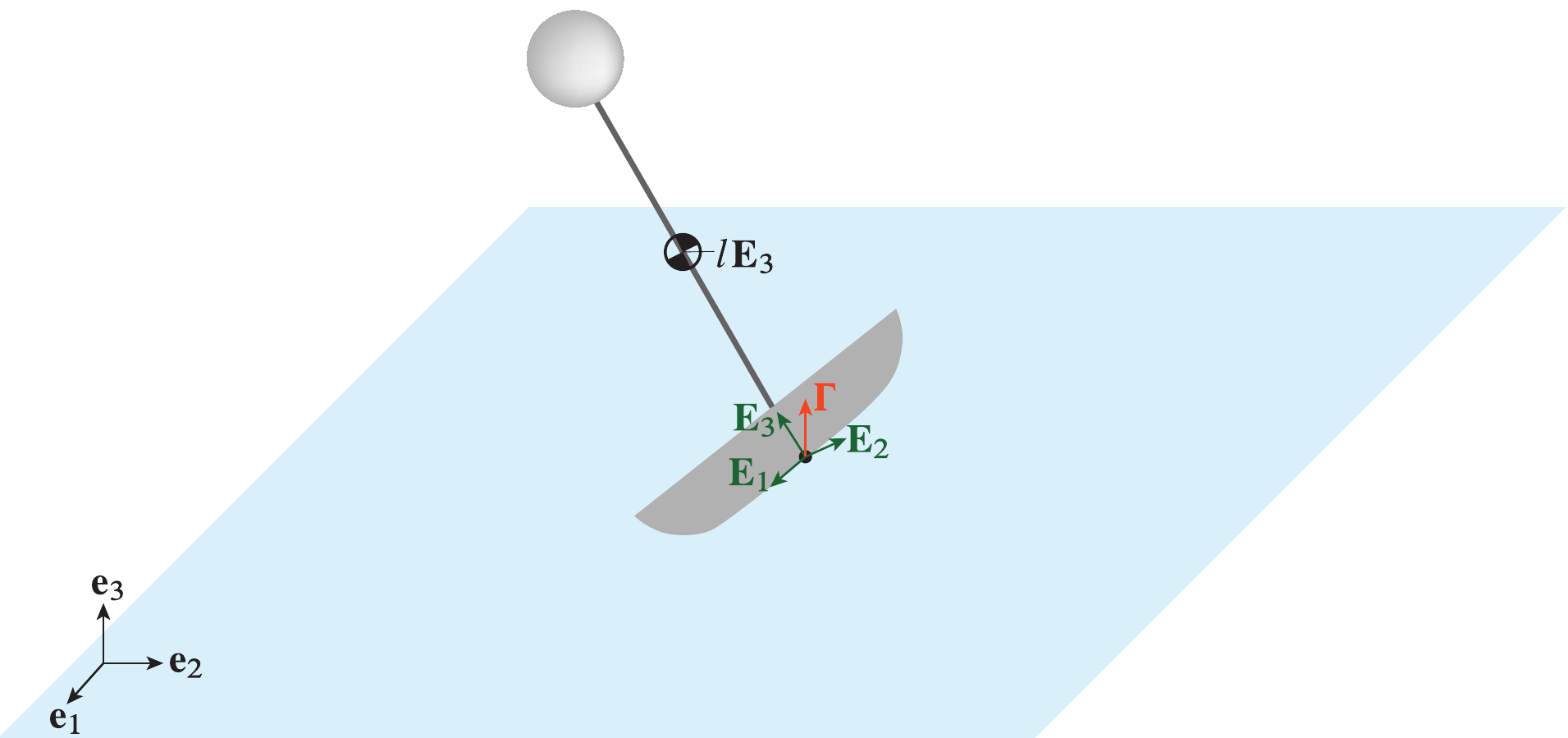}
  \caption{What we call the pendulum skate here is a simple model for a figure skater developed by \citet{GzPu2020}.
    The unit vector $\bGamma$---the vertical upward direction \textit{seen from the body frame} $\{\mathbf{E}_{i}\}_{i=1}^{3}$---is the advected parameter here.}
  \label{fig:PendulumSkate}
\end{figure}

Following \cite{GzPu2020} (see also \Cref{ssec:PendulumSkate} below), the configuration space is the semidirect product Lie group $\SE(3) \defeq \SO(3)\ltimes\R^3$, or the matrix group
\begin{equation}
  \label{eq:SE3}
  \SE(3) = \setdef{
    (R,\bx) \defeq
    \begin{bmatrix}
      R & \bx \\
      \bzero^{T} & 1
    \end{bmatrix}
  }{ R \in \SO(3),\, \bx \in \R^{3} }.
\end{equation}
The gravity breaks the $\SE(3)$-symmetry of the pendulum skate, just as in the well-known example of the heavy top in the semidirect product theory of mechanics~\cite{MaRaWe1984a,MaRaWe1984b,HoMaRa1998a,CeHoMaRa1998}.
Hence \citet{GzPu2020} used the unit vector $\bGamma$---the vertical upward direction seen from the body frame as depicted in \Cref{fig:PendulumSkate} as in the heavy top equations---as an advected parameter vector, and derived its equations of motion as nonholonomic Euler--Poincar\'e equations with advected parameters on $\se(3) \times (\R^{3})^{*}$, where $\se(3)$ is the Lie algebra of $\SE(3)$ and the dual space $(\R^{3})^{*}$ of $\R^{3}$ is for advected parameters $\bGamma$ that take care of the broken $\SE(3)$-symmetry.

It turns out that it is not just the gravity that breaks the symmetry.
The constraints imposed by the rink---including the nonholonomic constraint that the skate cannot slide sideways---also break the symmetry as well.
We shall treat this in detail later, but here is an intuitive explanation of why it breaks the symmetry:
When we start off with the configuration space $\SE(3)$ without the gravity nor the rink, we have an ambient space without any preferred direction or orientation---hence the $\SE(3)$-symmetry.
However, by introducing the rink to the setting, we effectively introduce a special direction---the unit normal vector to the rink---to the ambient space, thereby breaking the $\SE(3)$-symmetry.
This is just like how the gravity breaks the symmetry by introducing the special ``vertical'' direction to the ambient space
that would be otherwise uniform in any direction.

The main motivation for this work is to stabilize nonholonomic mechanical systems on Lie groups with such broken symmetry.
We shall show in \Cref{sec:PendulumSkate} that, for the pendulum skate, the upright spinning and upright sliding motions---ubiquitous in figure skating---are equilibria of the nonholonomic Euler--Poincar\'e equations derived in \cite{GzPu2020}.
As the intuition suggests, the spinning equilibrium is stable only under certain conditions, whereas the sliding equilibrium is always unstable.
We note in passing that these are equilibria of the reduced system and hence are relative equilibria of the original equations of motion.

Motivated by the problem of finding a control law to stabilize the sliding equilibrium, we would like to extend the method of Controlled Lagrangians of \citet{BlLeMa2001} to the Euler--Poincar\'e equations with symmetry-breaking nonholonomic constraints.
Particularly, our main goal is to build on the nonholonomic Euler--Poincar\'e theory of \citet{Sc2002} (see also \citet[Section~12.3]{Ho2011b}) to derive matching conditions for the Controlled Lagrangians applied to such systems.

The pendulum skate indeed necessitates a slight generalization of \cite{Sc2002} because it does not fit into the most general setting of \cite[Section~2.1 and Theorem~1]{Sc2002}.
We note that \citet{GaYo2015} made such a generalization in a more abstract and general Dirac structure setting.
Our focus is rather on first having a concrete expression for the nonholonomic Euler--Poincar\'e equations, particularly on a systematic way to eliminate the Lagrange multipliers in the equations of motion arising from the constraints.

Developing a general method of controlled Lagrangians for nonholonomic systems is challenging, particularly because of the Lagrange multipliers.
Indeed, extensions of the method of Controlled Lagrangians to nonholonomic systems are limited to a very special class of Lagrange--d'Alembert equations; see, e.g., \citet{ZeBlLeMa2000a,ZeBlLeMa2000b,ZeBlMa2002a,ZeBlMa2002b}.
To our knowledge, an extension to nonholonomic Euler--Poincar\'e equations has been done by \citet{Sc2002} only for the so-called Chaplygin top.

\subsection{Main Results and Outline}
We build on the work of \citet{Sc2002} to develop matching conditions for mechanical systems on Lie group $\mathsf{S}$ with symmetry-breaking nonholonomic constraints.
Particularly, we assume: (i)~The left $\mathsf{S}$-invariance of the Lagrangian is broken, but can be recovered using an advected parameter $\Gamma$ (sometimes called a Poisson vector);
(ii)~the nonholonomic constraints are not left-invariant either, but this broken symmetry is also recovered by using the same $\Gamma$.
We shall explain the details of this setting in \Cref{sec:broken_symmetry} using the pendulum skate as an example.

In \Cref{sec:NHEP}, we formulate the reduced Lagrange--d'Alembert principle and derive the nonholonomic Euler--Poincar\'e equations with broken symmetry as \Cref{prop:NHEP}, giving a generalization of the General Theorem of \citet[Theorem~1]{Sc2002}.
However, as mentioned earlier, it is a special case of the Dirac reduction for nonholonomic systems by \citet{GaYo2015}, and is included here only for completeness.

The ideas that nonholonomic constraints may be symmetry-breaking and that the broken symmetry may be recovered using an advected parameter are not new; they were discussed in \cite{Sc2002} as well as \citet{Ta2004}, \citet{BuBu2016}, and \citet{Bu2018}.
Our treatment is more systematic as in \cite{GaYo2015} and applies to more general constraints than those of \cite{Sc2002}.

The above result leads to the notion of $\Gamma$-dependent quasivelocities in \Cref{sec:quasivelocities}.
Quasivelocities have been often used in nonholonomic mechanics; see, e.g., \citet{EhKoMoRi2004,GrLeMaMa2009,BlMaZe2009,BaZeBl2012,Ze2016,ZeLeBl2012,BoMa2015} and references therein.
However, to our knowledge, the idea of advected-parameter-dependent quasivelocities is new.
Those $\Gamma$-dependent quasivelocities help us eliminate the Lagrange multipliers in the nonholonomic Euler--Poincar\'e equations derived in \Cref{prop:NHEP}.

In \Cref{sec:PendulumSkate}, we apply the result from \Cref{sec:quasivelocities} to the pendulum skate, find a family of equilibria including the upright spinning and upright sliding ones, and analyze their stability, specifically the Lyapunov stability of the equilibrium of the reduced system.

In \Cref{sec:ControlledLagrangian}, we show that the nonholonomic Euler--Poincar\'e equations written in the $\Gamma$-dependent quasivelocities help us extend the Controlled Lagrangians of \citet{BlLeMa2001}---developed for the \textit{standard} Euler--Poincar\'e equation---to our nonholonomic setting.
Indeed, the derivation and expressions of the resulting matching conditions in \Cref{prop:matching} are almost the same as those in \cite{BlLeMa2001} thanks to the formulation using the quasivelocities.
The result also generalizes the simpler and ad-hoc matching of \cite{Sc2002} for the Chaplygin top and our own for the pendulum skate in \cite{GaOh2022}.

Finally, in \Cref{sec:stabilization}, we apply the Controlled Lagrangian to find a feedback control to stabilize the sliding equilibrium, and illustrate the result in a numerical simulation.

\section{Broken Symmetry in Lagrangian and Nonholonomic Constraints}
\label{sec:broken_symmetry}
In this section, we would like to use the pendulum skate of \citet{GzPu2020} to describe the basic ideas behind mechanical systems on Lie groups with symmetry-breaking nonholonomic constraints.

\subsection{Pendulum Skate}
\label{ssec:PendulumSkate}
Let $\{\be1,\be2,\be3\}$ and $\{\bE1,\bE2,\bE3\}$ be the spatial and body frames, respectively, where the body frame is aligned with the principal axes of inertia; particularly $\bE1$ is aligned with the edge of the blade as shown in \Cref{fig:PendulumSkate}.
The two frames are related by the rotation matrix $R(t) \in \SO(3)$ whose column vectors represent the body frame viewed in the spatial one at time $t$.

The origin of the body frame is the blade-ice contact point, and has position vector $\bx(t)=(x_1(t),x_2(t),0)$ at time $t$ in the spatial frame.
However, we shall treat the position $\bx$ of the contact point as a vector in $\R^{3}$ and impose $x_{3} = 0$ as a constraint.
Hence, as mentioned earlier, the configuration space is the semidirect product Lie group $\SE(3) \defeq \SO(3) \ltimes \R^{3}$ from \eqref{eq:SE3}, where the multiplication rule is given by
\begin{equation*}
  (R_{0}, \bx_{0}) \cdot (R, \bx) = (R_{0} R, R_{0}\bx + \bx_{0}).
\end{equation*}

Let us find the Lagrangian of the system.
If $t \mapsto s(t) = (R(t),\bx(t))$ is the dynamics of the system in $\SE(3)$, then
\begin{equation*}
  s^{-1} \dot{s} =
  \begin{bmatrix}
    R^{T} \dot{R} & R^{T}\dot{\bx} \\
    \bzero^{T} & 0
  \end{bmatrix}
  \eqdef
  \begin{bmatrix}
    \hat{\Omega} & \bY \\
    \bzero^{T} & 0
  \end{bmatrix}
  \eqdef (\hat{\Omega}, \bY)
  \in \se(3),
\end{equation*}
where $\hat{\Omega} \defeq R^{T} \dot{R}$ is the body angular velocity; $\bY \defeq R^{T} \dot{\bx}$ is the velocity of the blade-ice contact point seen from the body frame.

Suppose that the center of mass is located at ${l\bE3}$ in the body frame as shown in \Cref{fig:PendulumSkate}, and hence, in the spatial frame, is located at $\bx_{\text{cm}}(t) \defeq \bx(t) + l R(t) \bE3$.
Then, the Lagrangian $L_{\be3}\colon T\SE(3)\to\R$ is given by:
\begin{align*}
  L_{\be3}\parentheses{ R, \bx, \dot{R}, \dot{\bx} }
  &\defeq \frac{1}{2}\tr\parentheses{ \dot{R} \mathbb{J} \dot{R}^{T} }
    + \frac{m}{2}\norm{\dot{\bx}_{\text{cm}}}^2
    - m \mathrm{g} \be3^{T} \bx_{\text{cm}}
  \\
  &= \frac{1}{2} \tr\parentheses{ R^{T} \dot{R} \mathbb{J} \dot{R}^{T} R }
    + \frac{m}{2} \norm{\dot{\bx}+\dot{R}l\bE3 }^{2}
    - m \mathrm{g} l \be3^{T} R \bE3
  \\
  &= \frac{1}{2}\tr\parentheses{ \hat{\Omega} \mathbb{J} \hat{\Omega}^{T} }
    + \frac{m}{2}\norm{\bY+l\hat{\Omega}\bE3}^2
    - m \mathrm{g} l (R^{T}\be3)^{T} \bE3,
\end{align*}
where $m$ is the total mass, $\mathrm{g}$ is the gravitational acceleration, $\|\cdot\|$ is the Euclidean norm, and $\mathbb{J}$ is the inertia matrix.

We also identify ${\bOmega=(\Omega_1, \Omega_2, \Omega_3)\in\R^3}$ with $\hat{\Omega}\in\so(3)$ via the hat map~\cite[\S5.3]{HoScSt2009}:
\begin{equation*}
  \widehat{(\,\cdot\,)}\colon{\R}^3\to\se(3);
  \qquad
  \bOmega \mapsto 
  \hat{\Omega} = \begin{bmatrix}
    0 &\mkern-15mu -\Omega_3 & \Omega_2 \\
    \mkern8mu\Omega_3 &\mkern-6mu 0 &\mkern-8mu -\Omega_1 \\
    -\Omega_2 & \Omega_1 & 0
  \end{bmatrix}.
\end{equation*}
Then we have the following correspondence with the cross product: $\hat{\Omega}\mathbf{y}= \bOmega\times\mathbf{y}$ for every $\mathbf{y}\in\R^3$.
So we may use ${(\bOmega,\bY)\in\R^3\times\R^3}$ as coordinates for $\se(3) \cong \R^{3} \times \R^{3}$, and we have
\begin{equation*}
  L_{\be3}\parentheses{ R, \bx, \dot{R}, \dot{\bx} }
  = \frac{1}{2} \bOmega^{T} \mathbb{I} \bOmega
  + \frac{m}{2}\norm{ \bY + l \bOmega \times \bE3 }^2
  - m \mathrm{g} l (R^{T}\be3)^{T} \bE3,
\end{equation*}
where $\mathbb{I} \defeq \tr(\mathbb{J}) \mathsf{I} -\mathbb{J} = \diag(I_{1}, I_{2}, I_{3})$ with $\mathsf{I}$ being the $3\times3$ identity matrix is the (body) moment of inertia tensor; see, e.g., \citet[\!\S7.1]{HoScSt2009}.

The Lagrangian does \textit{not} possess the full $\SE(3)$-symmetry due to the gravity.
Indeed, if $R_{0}^{T} \be3 \neq \be3$ then, in general,
\begin{equation}
  \label{eq:broken_symmetry-PendulumSkate}
  L_{\be3}\parentheses{ R_{0}R, R_{0}\bx + \bx_{0}, R_{0}\dot{R}, R_{0}\dot{\bx} }
  \neq
  L_{\be3}\parentheses{ R, \bx, \dot{R}, \dot{\bx} },
\end{equation}
although the equality holds if $R_{0}^{T} \be3 = \be3$.
In the next subsection, we shall discuss a general theory on how to recover the full symmetry of such a Lagrangian with broken symmetry, and come back to the pendulum skate as an example.

\subsection{Broken Symmetry in Lagrangian}
\label{ssec:Lagrangian}
Let $L_{\gamma_{0}}\colon T\mathsf{S} \to \R$ be a Lagrangian with a fixed parameter $\gamma_{0} \in X^{*}$, where $X^{*}$ is the dual of a vector space $X$.
For each $s_{0} \in \mathsf{S}$, let $\mathsf{L}_{s_{0}}\colon \mathsf{S} \to \mathsf{S}; s \mapsto s_{0} s$ be the left translation, and $T\mathsf{L}_{s_{0}}\colon \dot{s} \mapsto T\mathsf{L}_{s_{0}}(\dot{s}) \eqdef s_{0} \dot{s}$ be its tangent lift.

Suppose that the Lagrangian $L_{\gamma_{0}}$ does not possess $\mathsf{S}$-symmetry but that we may recover the symmetry as follows:
\begin{assumption}
  \label{asm:1}
  \leavevmode
  \begin{enumerate}[(i)]
  \item The Lagrangian $L_{\gamma_{0}}$ is not $\mathsf{S}$-invariant: $L_{\gamma_{0}} \circ T\mathsf{L}_{s_{0}} \neq L_{\gamma_{0}}$ for some $s_{0} \in \mathsf{S}$.
    \smallskip
  \item There exist an extended Lagrangian $L\colon T\mathsf{S} \times X^{*} \to \R$ and a representation
    \begin{equation}
      \label{eq:S-rep_on_X}
      \mathsf{S} \times X \to X;
      \qquad
      (s, \mathsf{x}) \mapsto s \mathsf{x}.
    \end{equation}
    such that the following are satisfied:
    \smallskip
    \begin{itemize}
    \item $L(\,\cdot\,, \gamma_{0}) = L_{\gamma_{0}}$.
      \smallskip
    \item With the induced representation $(s, \gamma) \mapsto s \gamma$ on $X^{*}$, i.e.,
    \begin{equation}
      \label{eq:S-rep_on_X*}
      \mathsf{S} \times X^{*} \to X^{*};
      \qquad
      (s, \gamma) \mapsto s \gamma
      \quad\text{such that}\quad
      \ip{ s \gamma }{ \mathsf{x} } = \ip{ \gamma }{ s^{-1} \mathsf{x} }
      \quad
      \forall \mathsf{x} \in X,
    \end{equation}
    the (left) $\mathsf{S}$-symmetry of the Lagrangian is recovered:
    \begin{equation*}
      L( s_{0}s,  s_{0}\dot{s}, s_{0}\gamma ) = L(s, \dot{s}, \gamma)
      \qquad
      \forall s_{0}, s \in \mathsf{S}
      \quad
      \forall \gamma \in X^{*}.
    \end{equation*}
    \end{itemize}
  \end{enumerate}
  As a result, we may define the reduced Lagrangian as follows:
  \begin{equation}
    \label{eq:ell}
    \ell\colon \mathfrak{s} \times X^{*} \to \R;
    \qquad
    \ell(\xi, \Gamma) \defeq L(e, \xi, \Gamma),
  \end{equation}
  where $\mathfrak{s}$ is the Lie algebra of $\mathsf{S}$ and $e \in \mathsf{S}$ is the identity.
\end{assumption}

\begin{remark}[Roles of advected parameter $\Gamma$]
  \label{rem:role_of_Gamma}
  \leavevmode
  \begin{enumerate}[(i)]
  \item Note that our original problem had the Lagrangian $L_{\gamma_{0}}(s, \dot{s}) = L(s, \dot{s}, \gamma_{0})$, but then the above $\mathsf{S}$-symmetry implies that this is in turn equal to $L(e, s^{-1}\dot{s}, s^{-1}\gamma_{0})$.
    Hence, in view of \eqref{eq:ell}, the advected parameter $\Gamma$ and the (static) parameter $\gamma_{0}$ are related as $\Gamma = s^{-1}\gamma_{0}$.
    \smallskip
  \item In exchange of recovering the $\mathsf{S}$-symmetry, the phase space is extended by attaching $X^{*}$ to $\mathfrak{s}$, hence the reduced equations involve both $\xi$ and $\Gamma$ as we shall see in \Cref{prop:NHEP}.
  \end{enumerate}
\end{remark}

\begin{example}[Lagrangian of pendulum skate~\cite{GzPu2020}]
  \label{ex:Lagrangian-PendulumSkate}
  For the pendulum skate from \Cref{ssec:PendulumSkate}, we define the extended Lagrangian $L\colon T^{*}\SE(3) \times (\R^{3})^{*} \to \R$ as follows:
  \begin{equation*}
    L\parentheses{ R, \bx, \dot{R}, \dot{\bx}, \bgamma }
    = \frac{1}{2} \bOmega^{T} \mathbb{I} \bOmega
    + \frac{m}{2}\norm{ \bY + l \bOmega \times \bE3 }^2
    - m \mathrm{g} l (R^{T}\bgamma)^{T} \bE3
  \end{equation*}
  so that $L\parentheses{ R, \bx, \dot{R}, \dot{\bx}, \be3 } = L_{\be3}\parentheses{ R, \bx, \dot{R}, \dot{\bx} }$.

  We also define an $\SE(3)$-representation on $\R^{3}$ by setting $(R, \bx) \mathbf{y} \defeq R \mathbf{y}$.
  Then, identifying $(\R^{3})^{*}$ with $\R^{3}$ via the dot product, we have
  \begin{equation*}
    \parentheses{ (R, \bx) \bgamma } \cdot \mathbf{y}
    = \bgamma \cdot \parentheses{ (R, \bx)^{-1} \mathbf{y} }
    = \bgamma \cdot \parentheses{ R^{-1} \mathbf{y} }
    = (R \bgamma ) \cdot \mathbf{y}.
  \end{equation*}
  Therefore, we have
  \begin{equation}
    \label{eq:SE3-action_on_gamma}
    \SE(3) \times (\R^{3})^{*} \to (\R^{3})^{*};
    \qquad
    ((R, \bx), \bgamma) \mapsto R \bgamma
  \end{equation}
  using the standard matrix-vector multiplication.

  Then we see that the $\SE(3)$-symmetry is recovered:
  For every $(R_{0}, \bx_{0}), (R, \bx) \in \SE(3)$ and every $\bgamma \in \R^{3}$,
  \begin{equation*}
    L\parentheses{ R_{0}R, R_{0}\bx + \bx_{0}, R_{0}\dot{R}, R_{0}\dot{\bx}, R_{0}\bgamma }
    = L\parentheses{ R, \bx, \dot{R}, \dot{\bx}, \bgamma }.
  \end{equation*}
  Therefore, we may define the reduced Lagrangian
  \begin{equation}
    \label{eq:ell-PendulumSkate}
    \begin{split}
      &\ell \colon \se(3) \times (\R^{3})^{*} \cong \R^{3} \times \R^{3} \times \R^{3} \to \R; \\
      &\ell(\bOmega, \bY, \bGamma) \defeq \frac{1}{2} \bOmega^{T} \mathbb{I} \bOmega
      + \frac{m}{2}\norm{ \bY + l \bOmega \times \bE3 }^2
    - m \mathrm{g} l \bGamma^{T} \bE3,
    \end{split}
  \end{equation}
  which agrees with \cite[Eq.~(1)]{GzPu2020}.
  Since $\gamma_{0} = \be3$ for our Lagrangian, the advected parameter is $\bGamma = R^{T} \gamma_{0} = R^{T} \be3$, which gives the vertical upward direction (essentially the direction of gravity) seen from the body frame; see \Cref{fig:PendulumSkate}.
\end{example}

\begin{remark}[Intuition behind symmetry recovery]
  Here is an intuitive interpretation of how the symmetry recovery works in \Cref{ex:Lagrangian-PendulumSkate} above.
  When we talked about the broken $\SE(3)$-symmetry in \eqref{eq:broken_symmetry-PendulumSkate}, the vertical upward direction $\be3$ (essentially the direction of gravity) was fixed, and the $\SE(3)$-action $(R_{0}, \bx_{0}) \cdot (R, \bx) = (R_{0} R, R_{0}\bx + \bx_{0})$ rotates the pendulum skate by $R_{0}$ and translates it by $\bx_{0}$ but \textit{left the direction of the gravity unchanged}, resulting in a system whose direction of gravity is different from the original configuration $(R,\bx)$; hence the symmetry is broken.
  On the other hand, when we introduced $\bgamma$ as a new variable in the Lagrangian above and let $\SE(3)$ act on $\gamma$ as defined in \eqref{eq:SE3-action_on_gamma}, we \textit{co-rotated the direction of the gravity along with the pendulum skate}, resulting in a system with the same relative direction of gravity as the original one; hence the whole system---now involving the variable direction of gravity---possesses the $\SE(3)$-symmetry.
\end{remark}

\subsection{Broken Symmetry in Nonholonomic Constraints}
\label{ssec:constraints}
Following \citet[Section~4.2]{GaYo2015}, we assume that the system is subject to nonholonomic constraints with a fixed parameter $\gamma_{0} \in X^{*}$, that is, the constraint is defined by a corank $r$ smooth distribution $\mathcal{D}_{\gamma_{0}} \subset T\mathsf{S}$ so that the dynamics $t \mapsto s(t) \in \mathsf{S}$ satisfies $\dot{s}(t) \in \mathcal{D}_{\gamma_{0}}(s(t))$.
In other words, we have 1-forms $\{ \Psi^{a}_{\gamma_{0}} \}_{a=1}^{r}$ on $\mathsf{S}$ such that the annihilator $\mathcal{D}_{\gamma_{0}}^{\circ}(s) = \Span\{ \Psi^{a}_{\gamma_{0}}(s) \}_{a=1}^{r}$, i.e., $\dot{s}(t) \in \ker \Psi^{a}_{\gamma_{0}}(s(t))$ for any $a \in \{1, \dots, r \}$.

Let us now assume that $\mathcal{D}_{\gamma_{0}}$ does not possess $\mathsf{S}$-symmetry but that we may recover the symmetry in a similar way as above:
\begin{assumption}
  \label{asm:2}
  \leavevmode
  \begin{enumerate}[(i)]
  \item $\mathcal{D}_{\gamma_{0}}$ is not $\mathsf{S}$-invariant: $(\mathsf{L}_{s_{0}})_{*}\mathcal{D}_{\gamma_{0}} \neq \mathcal{D}_{\gamma_{0}}$ and hence $\mathsf{L}_{s_{0}}^{*} \Psi^{a}_{\gamma_{0}} \neq \Psi^{a}_{\gamma_{0}}$ for some $s_{0} \in \mathsf{S}$.
    %
    \smallskip
  \item There exist $r$ 1-forms on $\mathsf{S}$ with a parameter in $X^{*}$, i.e.,
    \begin{equation*}
      \Psi^{a}\colon \mathsf{S} \times X^{*} \to T^{*}\mathsf{S}
      \quad\text{with}\quad
      \Psi^{a}(s, \gamma) \in T_{s}^{*}\mathsf{S}
      \quad
      \forall s \in \mathsf{S}
      \quad
      \forall \gamma \in X^{*}
      \quad
      \forall a \in \{1, \dots, r \}
    \end{equation*}
    such that the following are satisfied:
    \smallskip
    \begin{itemize}
    \item $\Psi^{a}(\,\cdot\,, \gamma_{0}) = \Psi^{a}_{\gamma_{0}}$ for every $a \in \{1, \dots, r \}$.
      \smallskip
    \item Using the same $\mathsf{S}$-representation on $X^{*}$ from \Cref{asm:1}, the (left) $\mathsf{S}$-symmetry of the 1-forms is recovered:
      \begin{equation}
        \label{eq:A-symmetry}
        (\mathsf{L}_{s_{0}}^{*} \Psi^{a})(\,\cdot\,, s_{0} \gamma) = \Psi^{a}( \,\cdot\,, \gamma )
        \qquad
        \forall s_{0} \in \mathsf{S}
        \quad
        \forall \gamma \in X^{*}
        \quad
        \forall a \in \{1, \dots, r \}.
      \end{equation}
    \end{itemize}
  \end{enumerate}
  As a result, we have the $\gamma$-dependent distribution
  \begin{equation*}
    \mathcal{D}(s,\gamma) \defeq \bigcap_{a=1}^{r} \ker \Psi^{a}(s,\gamma) \subset T_{s}\mathsf{S}
    \qquad
    \forall s \in \mathsf{S}
    \quad
    \forall \gamma \in X^{*}
  \end{equation*}
  with $\mathsf{S}$-symmetry:
  \begin{equation*}
    T_{s}\mathsf{L}_{s_{0}}( \mathcal{D}(s, \gamma) ) =  \mathcal{D}( s_{0}s, s_{0}\gamma )
    \quad
    \forall s, s_{0} \in \mathsf{S}
    \quad
    \forall \gamma \in X^{*}.
  \end{equation*}
  We may then define the following $\Gamma$-dependent subspace in $\mathfrak{s}$ and $\Gamma$-dependent elements in $\mathfrak{s}^{*}$:
  \begin{equation}
    \label{eq:a-def}
    \mathfrak{d}(\Gamma) \defeq \mathcal{D}(e,\Gamma) \subset \mathfrak{s},
    \qquad
    \psi^{a}(\Gamma) \defeq \Psi^{a}(e, \Gamma) \in \mathfrak{s}^{*}
    \quad
    \forall a \in \{1, \dots, r \}
  \end{equation}
  so that the annihilator $\mathfrak{d}^{\circ}(\Gamma) = \Span\{ \psi^{a}(\Gamma) \}_{a=1}^{r} \subset \mathfrak{s}^{*}$.
\end{assumption}

\begin{remark}
  Note that the original distribution is $\mathcal{D}_{\gamma_{0}}(s) = \mathcal{D}(s, \gamma_{0})$, but then the above $\mathsf{S}$-symmetry implies that this is in turn equal to $T_{e}\mathsf{L}_{s}(\mathcal{D}(e, s^{-1}\gamma_{0}))$.
  Upon setting $\Gamma = s^{-1}\gamma_{0}$ as mentioned in \Cref{rem:role_of_Gamma}, one sees that this is the left translation of the subspace $\mathfrak{d}(\Gamma) \defeq \mathcal{D}(e,\Gamma) = \bigcap_{a=1}^{r} \ker \Psi^{a}(e,\Gamma)$ in $\mathfrak{s}$.
\end{remark}

\begin{example}[No-side-sliding constraint of pendulum skate]
  \label{ex:no-side-sliding}
  The skate blade moves without friction, but with a constraint that prohibits motions perpendicular to its edge.
  This means that the spatial velocity $\dot{\bx}$ has no components in the direction perpendicular to the plane spanned by $\{R\bE1,\be3\}$, i.e., $\ip{\dot{\bx}}{R\bE1\times\be3}=0$ as shown in \Cref{fig:NoSideSliding}.
  In other words, setting $\Psi_{\be3}(R,\bx) = (R \bE1 \times \be3)^{T} \d\bx$, one may write the constraint as $\dot{\bx} \in \ker \Psi_{\be3}$.
  \begin{figure}[hbtp]
    \centering
    \includegraphics[width=0.8
    \linewidth]{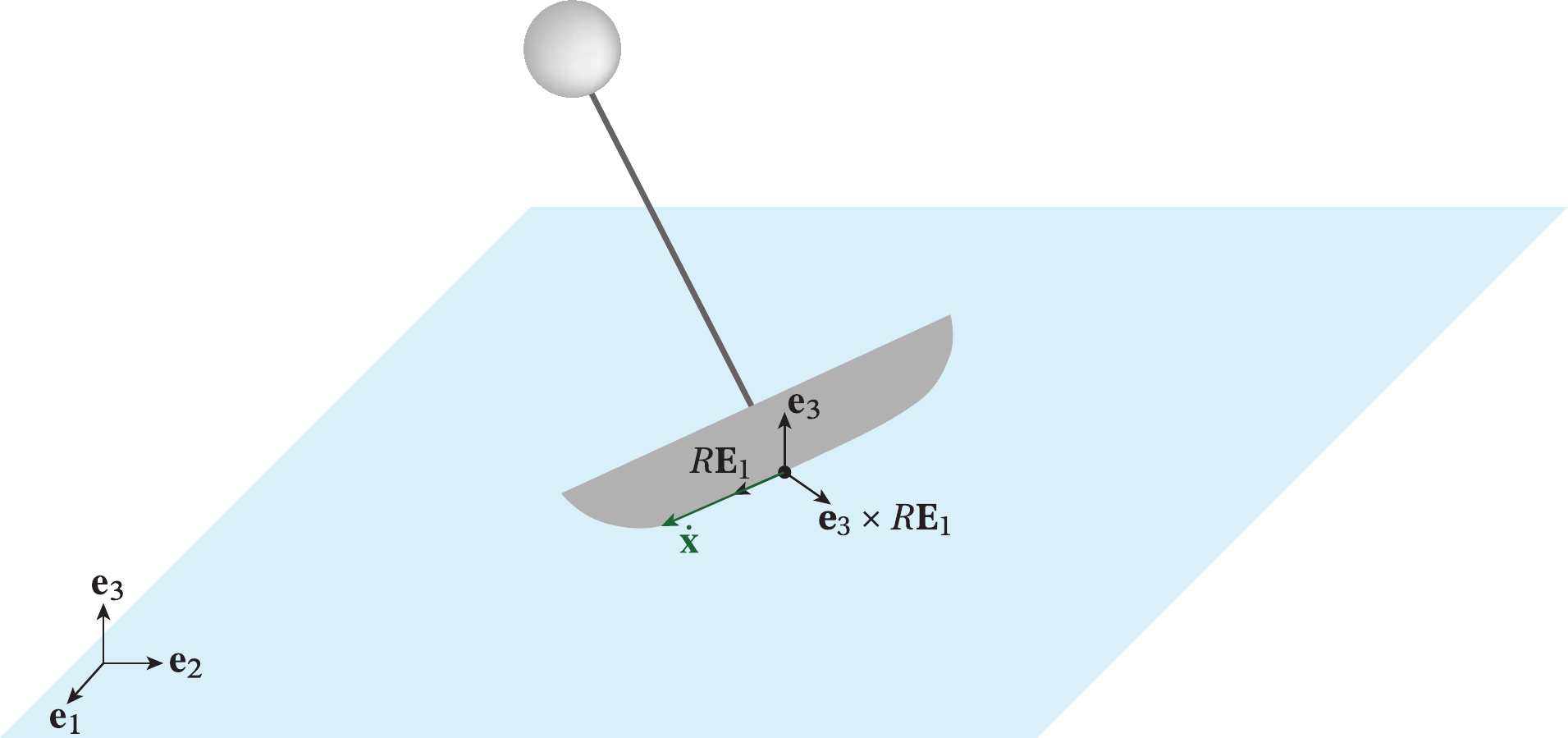}
    \caption{No side sliding.}
    \label{fig:NoSideSliding}
  \end{figure}
  It is easy to see that the $\SE(3)$-symmetry is broken:
  If we take $(R_{0},\bx_{0}) \in \SE(3)$ with $R_{0}^{T} \be3 \neq \be3$, then
  \begin{align*}
    (\mathsf{L}_{(R_{0},\bx_{0})}^{*} \Psi_{\be3})(R,\bx)
    &= T_{(R,\bx)}^{*}\mathsf{L}_{(R_{0},\bx_{0})} \parentheses{ \Psi_{\be3}(R_{0}R,R_{0}\bx) }   \\
    &= ((R_{0} R \bE1) \times \be3)^{T} \d(R_{0}\bx) \\
    &= \parentheses{ (R \bE1) \times (R_{0}^{T}\be3) }^{T} \d\bx \\
    &\neq \Psi_{\be3}(R,\bx)
  \end{align*}
  in general, although $\mathsf{L}_{(R_{0},\bx_{0})}^{*} \Psi_{\be3} = \Psi_{\be3}$ if $R_{0}^{T} \be3 = \be3$.
  
  One notices that this is the same type of broken symmetry as in the Lagrangian from \Cref{ex:Lagrangian-PendulumSkate} caused by the gravity.
  This suggests the same remedy applied to the Lagrangian would recover the $\SE(3)$-symmetry of the constraint too.
  Indeed, define
  \begin{equation*}
    \Psi\colon \SE(3) \times (\R^{3})^{*} \to T^{*}\SE(3);
    \qquad
    \Psi((R,\bx), \bgamma) \defeq ((R \bE1) \times \bgamma)^{T} \d\bx
  \end{equation*}
  so that $\Psi((R,\bx), \be3) = \Psi_{\be3}(R,\bx)$.
  Then we see that, for every $(R_{0}, \bx_{0}), (R, \bx) \in \SE(3)$ and every $\bgamma \in \R^{3}$,
  \begin{align*}
    (\mathsf{L}_{(R_{0},\bx_{0})}^{*} \Psi)((R,\bx), R_{0}\bgamma)
    &= T_{(R,\bx)}^{*}\mathsf{L}_{(R_{0},\bx_{0})} \parentheses{ \Psi_{\be3}((R_{0}R,R_{0}\bx), R_{0}\gamma) } \\
     &= \parentheses{ (R_{0} R \bE1) \times (R_{0}\bgamma) }^{T} \d(R_{0}\bx) \\
    &=  \parentheses{ (R \bE1) \times \bgamma }^{T} \d\bx \\
    &= \Psi( (R, \bx), \bgamma).
  \end{align*}
  Therefore, we define the $\bGamma$-dependent element
  \begin{equation*}
    \psi^{3}(\bGamma) \defeq \Psi( (I, \bzero), \bGamma)
    = (\bzero, \bE1 \times \bGamma) \in \se(3)^{*} \cong \R^{3} \times \R^{3},
  \end{equation*}
  where we used the superscript $3$ because there are two other constraints as we shall see in \Cref{ex:PendulumSkate}.
\end{example}

\begin{remark}
  \label{rem:why_same_Gamma}
  Why does the same advected parameter $\bGamma$---which we used to recover the broken symmetry \textit{due to the gravity}---work in order to recover the broken symmetry \textit{due to the nonholonomic constraint} as well?
  The nonholonomic constraint is characterized by how one introduces the rink into the system.
  In this problem, we introduced the rink as a horizontal plane---perpendicular to the direction of gravity.
  So $\bGamma$ gives \textit{both} the direction of gravity \textit{and} the orientation of the rink seen from the body frame.
\end{remark}

\section{Nonholonomic Euler--Poincar\'e Equation with Advected Parameters}
\label{sec:NHEP}
This section gives a review of the reduced Lagrange--d'Alembert principle for mechanical systems with broken symmetry.
Our result gives a slight generalization of the General Theorem of \citet[Theorem~1]{Sc2002}.
However, as mentioned earlier, ours is a special case of \citet[Theorem~4.3]{GaYo2015} as well, and so we include this result only for completeness.

\subsection{Reduced Lagrange--d'Alembert Principle with Broken Symmetry}
We would like to find the reduced equations of motion exploiting the recovered symmetry.
For our setting, one needs to turn the variational principle of \citet[Theorem~3.1]{HoMaRa1998a} (see also \citet[Theorem~1.1]{CeHoMaRa1998}) into a Lagrange--d'Alembert-type incorporating the symmetry-breaking nonholonomic constraints.
We note that this is done in \citet[Theorem~1]{Sc2002} with a semidirect product Lie group $\mathsf{S} = \mathsf{G} \ltimes V$ assuming a special class of nonholonomic constraints.

Before describing our version of the reduced Lagrange--d'Alembert principle with broken symmetry, let us introduce some notation used in the result to follow.
For any curve $t \mapsto s(t) \in \mathsf{S}$, we define $t \mapsto \xi(t) \defeq s(t)^{-1} \dot{s}(t) \in \mathfrak{s}$; conversely, given $t \mapsto \xi(t) \in \mathfrak{s}$, we define $t \mapsto s(t) \in \mathsf{S}$ via $\dot{s}(t) = s(t) \xi(t)$ with the left translation and initial condition $s(0) = s_{0}$ with a certain fixed element $s_{0} \in \mathsf{S}$.
Similarly, for any (infinitesimal) variation $t \mapsto \delta s(t) \in T_{s(t)}\mathsf{S}$ of the curve $t \mapsto s(t)$, we define $t \mapsto \eta(t) \defeq s(t)^{-1} \delta s(t)$ and conversely as well.

We shall also use the Lie algebra representation
\begin{equation*}
  \mathfrak{s} \times X \to X;
  \qquad
  (\xi, \mathsf{x}) \mapsto \dzero{\eps}{\exp(\eps\xi) \mathsf{x}} \eqdef \xi \mathsf{x},
\end{equation*}
and 
\begin{equation}
  \label{eq:s-rep_on_X*}
  \mathfrak{s} \times X^{*} \to X^{*};
  \qquad
  (\xi, \gamma) \mapsto \xi \gamma
  \quad\text{so that}\quad
  \ip{ \xi \gamma }{ \mathsf{x} } = -\ip{ \gamma }{ \xi \mathsf{x} }
  \quad
  \forall \mathsf{x} \in X,
\end{equation}
as well as the momentum map $\mathbf{K}\colon X \times X^{*} \to \mathfrak{s}^{*}$ defined by
\begin{equation}
  \label{eq:K}
  \ip{ \mathbf{K}(\mathsf{x},\Gamma) }{ \xi } = \ip{ \Gamma }{ \xi \mathsf{x} } = -\ip{ \xi \Gamma }{ \mathsf{x} }
  \quad
  \forall \xi \in \mathfrak{s},
\end{equation}
where $\xi \mathsf{x}$ is defined in the similar way as in \eqref{eq:s-rep_on_X*} using \eqref{eq:S-rep_on_X}.

Also, for any smooth function $f\colon E \to \R$ on a real vector space $E$, let us define its functional derivative $\tfd{f}{x} \in E^{*}$ at $x \in E$ such that, for any $\delta x \in E$, under the natural dual pairing $\ip{\,\cdot\,}{\,\cdot\,}\colon E^{*} \times E \to \R$,
\begin{equation*}
  \ip{ \fd{f}{x} }{ \delta x } = \dzero{\eps}{ f(x + \eps\delta x) }.
\end{equation*}

\begin{proposition}[Reduced Lagrange--d'Alembert Principle with Broken Symmetry]
  \label{prop:NHEP}
  Let $\gamma_{0} \in X^{*}$ be fixed, and suppose that the Lagrangian $L_{\gamma_{0}}\colon T\mathsf{S} \to \R$ and the distribution $\mathcal{D}_{\gamma_{0}} \subset T\mathsf{S}$ defining nonholonomic constraints satisfy \Cref{asm:1,asm:2} from \Cref{ssec:Lagrangian,ssec:constraints}, respectively.
  
  Then the following are equivalent:
  \begin{enumerate}[(i)]
  \item The curve $t \mapsto s(t) \in \mathsf{S}$ with the constraint $\dot{s}(t) \in \mathcal{D}_{\gamma_{0}}(s(t))$ satisfies the Lagrange--d'Alembert principle
    \begin{equation*}
      \delta \int_{t_{0}}^{t_{1}} L_{\gamma_{0}}\parentheses{ s(t), \dot{s}(t) }\,dt = 0
    \end{equation*}
    subject to $\delta s(t_{0}) = \delta s(t_{1}) = 0$ and $\delta s(t) \in \mathcal{D}_{\gamma_{0}}(s(t))$ for any $t \in (t_{0}, t_{1})$.
    \label{part:L-dA}
  \item The curve $t \mapsto \xi(t) \in \mathfrak{s}$ along with
    \begin{equation}
      \label{eq:Gamma}
      t \mapsto \Gamma(t) \defeq s(t)^{-1} \gamma_{0} \in X^{*}
    \end{equation}
    and with the constraint
    \begin{equation}
      \label{eq:zeta-constraint}
      \xi(t) \in \mathfrak{d}(\Gamma(t))
      \iff
      \ip{ \psi^{a}(\Gamma(t)) }{ \xi(t) } = 0
      \quad
      \forall t \in [t_{0}, t_{1}]
      \quad
      \forall a \in \{1, \dots, r \}
    \end{equation}
    satisfies the following reduced Lagrange--d'Alembert principle in terms of the reduced Lagrangian $\ell\colon \mathfrak{s} \times X^{*} \to \R$ defined in \eqref{eq:ell}:
     \begin{equation}
       \label{eq:reduced_LdA}
      \delta \int_{t_{0}}^{t_{1}} \ell( \xi(t), \Gamma(t) )\,dt = 0,
    \end{equation}
    where variations of $t \mapsto (\xi(t), \Gamma(t))$ are subject to the constraints
    \begin{equation}
      \label{eq:zeta_Gamma-constraint}
      \delta\xi = \dot{\eta} + \ad_{\xi} \eta,
      \qquad
      \delta\Gamma = -\eta\,\Gamma
    \end{equation}
    for every curve $t \mapsto \eta(t) \in \mathfrak{s}$ satisfying
    \begin{equation}
      \label{eq:eta-constraint}
      \delta\eta(t_{0}) = \delta\eta(t_{1}) = 0,
      \qquad
      \ip{ \psi^{a}(\Gamma(t)) }{ \eta(t) } = 0
      \quad
      \forall t \in [t_{0}, t_{1}]
      \quad
      \forall a \in \{1, \dots, r \}.
    \end{equation}
    \label{part:reduced_L-dA}
  \item The curve $t \mapsto (\xi(t),\Gamma(t))$ with $\xi(t) \in \mathfrak{d}(\Gamma(t))$ satisfies the nonholonomic Euler--Poincar\'e equations with advected parameter $\Gamma$:
    \begin{subequations}
      \label{eq:NHEP}
      \begin{align}
        \od{}{t}\parentheses{ \fd{\ell}{\xi} } &= \ad_{\xi}^{*} \fd{\ell}{\xi} + \mathbf{K}\parentheses{\fd{\ell}{\Gamma}, \Gamma} + \lambda_{a} \psi^{a}(\Gamma), \label{eq:NHEP1}
        \\
        \dot\Gamma &= -\xi \Gamma \label{eq:NHEP2}
      \end{align}
    \end{subequations}
    with $\Gamma(t_{0}) = s(t_{0})^{-1} \gamma_{0}$, where $\{ \lambda_{a} \}_{a=1}^{r}$ are Lagrange multipliers.
    \label{part:NHEP}
  \end{enumerate}
\end{proposition}
\begin{proof}
  Let us first show the equivalence between \eqref{part:L-dA} and \eqref{part:reduced_L-dA}.
  First, let us show that the two action integrals are equal.
  Indeed, using the definition \eqref{eq:ell} of $\ell$,
  \begin{align*}
    L_{\gamma_{0}}(s, \dot{s}) = L(s, \dot{s}, \gamma_{0}) = L\parentheses{ e, s^{-1}\dot{s}, s^{-1}\gamma_{0} }
    = \ell(\xi, \Gamma).
  \end{align*}
  It is a standard result in the Euler--Poincar\'e theory that all variations of $t \mapsto s(t)$ with fixed endpoints induce and are induced by variations of $t \mapsto \xi(t)$ of the form $\delta\xi = \dot{\eta} + \ad_{\xi} \eta$ with $\eta$ vanishing at the endpoint; see, e.g., \cite[Theorem~3.1 and Lemma~3.2]{HoMaRa1998a}.
  Also, $\delta\Gamma = -\eta\Gamma$ easily follows from the definitions \eqref{eq:s-rep_on_X*} and \eqref{eq:Gamma}.
  Finally, one can show the equivalence between the nonholonomic constraints on $\delta s$ and $\eta$ as follows:
  For every $a \in \{1, \dots, r \}$,
  \begin{align*}
    \delta s \in \mathcal{D}_{\gamma_{0}}(s)
    \iff
    0 &= \ip{ \Psi^{a}(s, \gamma_{0}) }{\delta s} \\
      &= \ip{ \Psi^{a}(s, \gamma_{0}) }{ s s^{-1}\delta s} \\
      &= \ip{ (L_{s}^{*}\Psi^{a})(e, \gamma_{0}) }{ \eta } \\
      &= \ip{ \Psi^{a}(e, s^{-1}\gamma_{0}) }{ \eta } \quad (\because \eqref{eq:A-symmetry})\\
      &= \ip{ \psi^{a}(\Gamma) }{ \eta } \quad (\because \text{\eqref{eq:a-def} and definition of $\Gamma$}).
  \end{align*}

  It remains to show the equivalence between \eqref{part:reduced_L-dA} and \eqref{part:NHEP}.
  As is done in the proof of \cite[Theorem~3.1]{HoMaRa1998a},
  \begin{align*}
    \delta \int_{t_{0}}^{t_{1}} \ell( \xi(t), \Gamma(t) )\,dt
    &= \int_{t_{0}}^{t_{1}} \parentheses{ \ip{ \fd{\ell}{\xi} }{ \delta\xi } + \ip{ \delta\Gamma }{ \fd{\ell}{\Gamma} } } dt \\
    &= \int_{t_{0}}^{t_{1}} \parentheses{ \ip{ \fd{\ell}{\xi} }{ \dot{\eta} + \ad_{\xi} \eta } - \ip{ \eta\Gamma }{ \fd{\ell}{\Gamma} } } dt \\
    &= -\int_{t_{0}}^{t_{1}}\ip{ \od{}{t}\parentheses{ \fd{\ell}{\xi} } - \ad_{\xi}^{*}\fd{\ell}{\xi} - \mathbf{K}\parentheses{\fd{\ell}{\Gamma}, \Gamma} }{ \eta } dt.
  \end{align*}
  Therefore, the reduced Lagrange--d'Alembert principle~\eqref{eq:reduced_LdA} with the nonholonomic constraints~\eqref{eq:eta-constraint} yields \eqref{eq:NHEP1}, whereas taking the time derivative of \eqref{eq:Gamma} yields \eqref{eq:NHEP2}.
  Conversely, it is clear that \eqref{eq:NHEP1} implies \eqref{eq:reduced_LdA} with the constraints imposed in \eqref{part:reduced_L-dA}, and integrating \eqref{eq:NHEP2} with the initial condition on $\Gamma$ given in \eqref{part:NHEP} yields \eqref{eq:Gamma}.
\end{proof}

\begin{remark}
  \label{rem:Lagrange_multipliers}
  The Lagrange multipliers $\{\lambda_{a}\}_{a=1}^{r}$ are determined by imposing the constraint $\ip{ \psi^{a}(\Gamma) }{ \xi } = 0$ for $a \in \{1, \dots, r\}$.
\end{remark}

\begin{example}[General Theorem of {\citet[Theorem~1]{Sc2002}}]
  Let $\mathsf{G}$ be a Lie group, $V$ be a vector space, and construct the semidirect product Lie group $\mathsf{S} \defeq \mathsf{G} \ltimes V$ under the multiplication
  \begin{align*}
    s_{1} \cdot s_{2}
    = (g_{1}, x_{1}) \cdot (g_{2}, x_{2})
    = (g_{1} g_{2}, g_{1} x_{2} + x_{1}),
  \end{align*}
  where $\mathsf{G} \times V \to V; (g,x) \mapsto g x$ is a representation.
  In what follows, we shall use the other induced $\mathsf{G}$- and $\mathfrak{g}$-representations on $V$ and $V^{*}$ defined in the same way we did for $X$ and $X^{*}$ above, as well as the associated momentum map $\mathbf{J}\colon V \times V^{*} \to \mathfrak{g}^{*}$ defined in the same way as $\mathbf{K}$ was in \eqref{eq:K}.
  Indeed, it is also assumed in \cite{Sc2002} that $X = V$ and the $\mathsf{S}$-representation on $X^{*}$ is defined as $(g,x) \gamma \defeq g \gamma$ using the $\mathsf{G}$-representation on $V^{*}$ induced by the above $\mathsf{G}$-representation on $V$.
  Hence, writing $\xi = (\Omega, Y) \in \mathfrak{s} = \mathfrak{g} \ltimes V$, it follows that $\xi \gamma = (\Omega, Y) \gamma = \Omega \gamma$ as well.
  It is then straightforward to see that, for every $(\Omega_{i}, Y_{i}) \in \mathfrak{s} = \mathfrak{g} \ltimes V$ with $i = 1, 2$,
  \begin{equation*}
    \ad_{(\Omega_{1},Y_{1})}(\Omega_{2},Y_{2})
    \defeq [(\Omega_{1},Y_{1}), (\Omega_{2},Y_{2})]
    = \parentheses{
      \ad_{\Omega_{1}}\!\Omega_{2},\,
      \Omega_{1} Y_{2} - \Omega_{2} Y_{1}
    },
  \end{equation*}

  What correspond to $\xi, \eta \in \mathfrak{s} = \mathfrak{g} \ltimes V$ in this example are
  \begin{align*}
    (\Omega, Y) &\defeq (g,x)^{-1} (\dot{g}, \dot{x}) = \parentheses{ g^{-1}\dot{g}, g^{-1}\dot{x} }, \\
    (\Sigma, Z) &\defeq (g,x)^{-1} (\delta{g}, \delta{x}) = \parentheses{ g^{-1}\delta{g}, g^{-1}\delta{x} },
  \end{align*}
  respectively.
  Then the constraint on $\xi$ in \eqref{eq:zeta_Gamma-constraint} becomes
  \begin{equation}
    \label{eq:deltazeta-Schneider}
    (\delta\Omega, \delta Y) = (\dot{\Omega}, \dot{Y}) + \parentheses{
      \ad_{\Omega}\!\Sigma,\,
      \Omega Z - \Sigma Y
    },
  \end{equation}

  In \cite{Sc2002}, the nonholonomic constraints are assumed to be in the form
  \begin{equation*}
    \ip{ \Psi((g,x), \gamma) }{ (\dot{g},\dot{x}) } = 
    \dot{x} - \bar{\Psi}((g, x), \gamma) \dot{g} = 0
    \quad\text{with}\quad
    \bar{\Psi}((g, x), \gamma) \colon T_{g}\mathsf{G} \to V
  \end{equation*}
  with the following (left) $\mathsf{S}$-invariance:
  \begin{equation*}
    \Psi((g_{0},x_{0})\cdot(g,x), (g_{0},x_{0})\gamma) = \Psi((g,x), \gamma)
    \quad
    \forall (g_{0},x_{0}), (g,x) \in \mathsf{S}
    \quad
    \forall \gamma \in X^{*}.
  \end{equation*}
  Hence we have
  \begin{equation*}
    \ip{ \psi(\Gamma) }{ (\Omega,Y) } \defeq Y - \bar{\psi}(\Gamma)\Omega
    \quad
    \text{with}
    \quad
    \bar{\psi}(\Gamma) \defeq \bar{\Psi}(e,\Gamma),
  \end{equation*}
  and thus the nonholonomic constraints \eqref{eq:zeta-constraint} and \eqref{eq:eta-constraint} applied to $(\Omega, Y)$ and $(\Sigma, Z)$ yield $Y = \bar{\psi}(\Gamma)\Omega$ and $Z = \bar{\psi}(\Gamma)\Sigma$.
  Substituting these expressions into the second equation in \eqref{eq:deltazeta-Schneider} yields
  \begin{equation*}
    \delta Y = \od{}{t}\parentheses{ \bar{\psi}(\Gamma)\Omega } + \Omega\,\bar{\psi}(\Gamma)\,\Sigma - \Sigma\,\bar{\psi}(\Gamma)\,\Omega,
  \end{equation*}
  which was in (2) of the General Theorem of \citet[Theorem~1]{Sc2002}.
\end{example}

\begin{example}[The Veselova system~\cite{VeVe1986}]
  Let $\mathsf{S} = \mathsf{G}$ (not a semidirect product).
  We still assume that the same broken \textit{left} symmetry and the recovery of \textit{left} symmetry for the Lagrangian $L_{\gamma_{0}}$ described in \Cref{ssec:Lagrangian}, but with $X = \mathfrak{g}$ so that $\gamma_{0} \in \mathfrak{g}^{*}$ and the adjoint representation $g \gamma \defeq \Ad_{g^{-1}}^{*} \gamma$ on $X^{*} = \mathfrak{g}^{*}$.
  We also assume that the constraint distribution $\mathcal{D} \subset T\mathsf{G}$ is corank 1, and is \textit{right} $\mathsf{G}$-invariant:
  \begin{equation*}
    T_{g}\mathsf{R}_{g_{0}}(\mathcal{D}(g)) = \mathcal{D}(g_{0} g)
    \quad
    \forall g, g_{0} \in \mathsf{G},
  \end{equation*}
  where $\mathsf{R}\colon \mathsf{G} \to \mathsf{G}$ is the right translation.
  So we may rewrite the constraint $\dot{g} \in \mathcal{D}(g)$ in terms of $\omega \defeq T_{g}\mathsf{R}_{g^{-1}}(\dot{g}) \eqdef \dot{g} g^{-1}$, i.e., the spatial angular velocity in the rigid body setting:
  \begin{equation*}
    \omega \in \mathfrak{d} \defeq \mathcal{D}(e).
  \end{equation*}
  We then further assume that $\mathfrak{d} = \ker \gamma_{0}$ with the same parameter $\gamma_{0} \in \mathfrak{g}^{*}$ for the Lagrangian, so that
  \begin{equation*}
    \omega \in \mathfrak{d} = \ker \gamma_{0}
    \iff
    \ip{ \gamma_{0} }{ \omega } = 0.
  \end{equation*}
  This is an example of the so-called nonholonomic LR systems~\cite{VeVe1986,FeJo2004,FeJo2009}.

  In short, the right invariant constraint is breaking the left invariance of the system.
  Indeed, noting that
  \begin{equation*}
    \omega = \dot{g} g^{-1} = g g^{-1} \dot{g} g^{-1} = \Ad_{g}\Omega
    \quad\text{with}\quad
    \Omega \defeq g^{-1} \dot{g},
  \end{equation*}
  we may rewrite the above constraint as follows:
  Setting $\Gamma = g^{-1} \gamma_{0} = \Ad_{g}^{*}\gamma_{0}$,
  \begin{equation*}
    \ip{ \gamma_{0} }{ \Ad_{g}\Omega } = \ip{ \Ad_{g}^{*} \gamma_{0} }{ \Omega } = \ip{ g^{-1} \gamma_{0} }{ \Omega } = 0
    \iff
    \ip{ \psi(\Gamma) }{ \Omega } = 0
    \quad
    \text{with}
    \quad
    \psi(\Gamma) \defeq \Gamma,
  \end{equation*}
  
  Particularly, with $\mathsf{G} = \SO(3)$ as in \cite{VeVe1986}, we have $X^{*} = \so(3)^{*}$, and so under the standard identification $\so(3) \cong \so(3)^{*} \cong \R^{3}$, we have
  \begin{equation*}
    \ad_{\bOmega}^{*}\bPi = \bPi \times \bOmega,
    \qquad
    \mathbf{K}(\mathbf{y}, \bGamma) = \mathbf{y} \times \bGamma,
    \qquad
    \bOmega\bGamma = \bOmega \times \bGamma,
  \end{equation*}
  and hence \eqref{eq:NHEP} gives
  \begin{equation*}
    \od{}{t}\parentheses{ \pd{\ell}{\bOmega} } = \pd{\ell}{\bOmega} \times \bOmega + \pd{\ell}{\bGamma} \times \bGamma + \lambda \bGamma, \\
    \qquad
    \dot{\bGamma} = \bGamma \times \bOmega
  \end{equation*}
  with constraint $\bGamma \cdot \bOmega = 0$.
  This gives the Veselova system~\cite{VeVe1986} with the standard kinetic minus potential form of the Lagrangian.
\end{example}

\begin{example}[Pendulum skate~\cite{GzPu2020}]
  \label{ex:PendulumSkate}
  As discussed in \Cref{ssec:PendulumSkate}, $\mathsf{S} = \SE(3) = \SO(3) \ltimes \R^{3}$ here.
  The system is subject to two more constraints in addition to the no-side-sliding constraint from \Cref{ex:no-side-sliding}~(see also \cite{GzPu2020}).
  The following constraints are actually holonomic as the derivations to follow suggest.
  However, we shall impose them as constraints on $\mathfrak{s}$ for the Euler--Poincar\'e formalism.
  \begin{itemize}
  \item \textit{Pitch constancy}:
    The blade does not rock back and forth.
    Specifically, the direction $R\bE1$ of the blade in the spatial frame is perpendicular to $\be3$ (see \Cref{fig:NoSideSliding}):
    \begin{equation}
      \label{eq:pitch_constancy}
      0 = (R\bE1)^{T}\be3
      = (R\bE1)^{T}R\bGamma
      = \bE1^{T}\bGamma = \Gamma_{1}.
    \end{equation}
    Taking the time derivative and using \eqref{eq:NHEP2} (which is $\dot{\bGamma} = \bGamma\times\bOmega$ here)
    \begin{equation*}
      0 = \bE1^{T} \dot{\bGamma}
        = \bE1^{T} (\bGamma\times\bOmega)
        = (\bE1\times\bGamma)^{T} \bOmega,
    \end{equation*}
    giving
    \begin{equation*}
      \psi^{1}(\bGamma) \defeq (\bE1\times\bGamma, \bzero) \in \se(3)^{*} \cong \R^{3} \times \R^{3}.
    \end{equation*}
  \item \textit{Continuous contact}:
    The skate blade is in permanent contact with the plane of the ice, i.e., $x_{3} = \be3^{T} \bx = 0$. Taking the time derivative,
    \begin{equation*}
      0 = \be3^{T} \dot{\bx}
      = (R^{T}\be3)^{T} R^{T}\dot{\bx}
      = \bGamma^{T} \bY,
    \end{equation*}
    giving
    \begin{equation*}
      \psi^{2}(\bGamma) \defeq (\bzero, \bGamma) \in \se(3)^{*} \cong \R^{3} \times \R^{3}.
    \end{equation*}
  \end{itemize}
  Combining the above two constraints with the no-side-sliding constraint from \Cref{ex:no-side-sliding}, we have
  \begin{equation*}
    \mathfrak{d}^{\circ}(\bGamma) = \Span\{ \psi^{a}(\bGamma) \}_{a=1}^{3} \subset \se(3)^{*}.
  \end{equation*}

  We also have, for every $(\bOmega,\bY) \in \se(3) \cong \R^{3} \times \R^{3}$ and every $(\bPi, \mathbf{P}) \in \se(3)^{*} \cong \R^{3} \times \R^{3}$,
  \begin{equation*}
    \ad_{(\bOmega, \bY)}^{*} (\bPi, \mathbf{P})
    = \parentheses{
      \bPi \times \bOmega + \mathbf{P} \times \bY,\,
      \mathbf{P} \times \bOmega
    },
  \end{equation*}
  and for every $(\mathbf{y}, \bGamma) \in X \times X^{*} \cong \R^{3} \times \R^{3}$,
  \begin{equation}
    \label{eq:K-PendulumSkate}
    \mathbf{K}(\mathbf{y},\bGamma)
    = (\mathbf{y} \times \bGamma, \bzero).
  \end{equation}
  Hence the nonholonomic Euler--Poincar\'e equations with advected parameter~\eqref{eq:NHEP} give
  \begin{equation}
    \label{eq:NHEP-PendulumSkate}
    \begin{split}
      \od{}{t}\parentheses{ \pd{\ell}{\bOmega} }
      &= \pd{\ell}{\bOmega} \times \bOmega
      + \pd{\ell}{\bY} \times \bY
      + \pd{\ell}{\bGamma} \times \bGamma
      + \lambda_{1} (\bE1\times\bGamma), \\
      \od{}{t}\parentheses{ \pd{\ell}{\bY} }
      &= \pd{\ell}{\bY} \times \bOmega + \lambda_{2} \bGamma + \lambda_{3} (\bE1 \times \bGamma), \\
      \dot{\bGamma} &= \bGamma \times \bOmega,
    \end{split}
  \end{equation}
  which is Eq.~(8) of \cite{GzPu2020}, and also gives Eq.~(9) of \cite{GzPu2020} with the reduced Lagrangian~\eqref{eq:ell-PendulumSkate}.
  As noted in \Cref{rem:Lagrange_multipliers}, the Lagrange multipliers $\{ \lambda_{a} \}_{a=1}^{3}$ are determined by imposing the constraints $\ip{ \psi^{a}(\Gamma) }{ \xi } = 0$ for $a \in \{1, 2, 3\}$; see \Cref{ex:no-side-sliding} for $\psi^{3}$.
\end{example}

\section{Eliminating Lagrange Multipliers}
\label{sec:quasivelocities}
One may algebraically find concrete expressions for the the Lagrange multipliers in the nonholonomic Euler--Poincar\'e equation~\eqref{eq:NHEP}.
However, they tend to be quite complicated, even with a rather simple Veselova system as shown in \citet[Eq.~(4.3)]{FeJo2004}.
Such a complication is detrimental when applying the method of Controlled Lagrangians, because it is difficult to ``match'' two equations if their structures are unclear.
In this section, we introduce $\Gamma$-dependent quasivelocities and systematically eliminate the Lagrange multipliers in \eqref{eq:NHEP}.

\subsection{$\Gamma$-dependent Hamel Basis}
We decompose the Lie algebra $\mathfrak{s}$ into the $\Gamma$-dependent constraints subspace $\mathfrak{d}(\Gamma)$ and its complement:
Setting $n \defeq \dim\mathfrak{s}$, 
\begin{equation*}
  \mathfrak{s} = \mathfrak{d}(\Gamma) \oplus \mathfrak{v}(\Gamma)
  \quad\text{with}\quad
  \mathfrak{d}(\Gamma) = \Span\{ \mathcal{E}_{\alpha}(\Gamma) \}_{\alpha=1}^{n-r},
  \quad
  \mathfrak{v}(\Gamma) = \Span\{ \mathcal{E}_{a}(\Gamma) \}_{a=1}^{r}.
\end{equation*}
Note that we are using Greek indices for $\mathfrak{d}$, whereas $a, b, c, \dots$ for $\mathfrak{v}$ and $i, j, k, \dots$ for the entire $\mathfrak{s}$.

So we may use Einstein's summation convention to write $\xi \in \mathfrak{s}$ as
\begin{equation*}
  \xi = v^{i}(\Gamma) \mathcal{E}_{i}(\Gamma)  = v^{\alpha}(\Gamma) \mathcal{E}_{\alpha}(\Gamma) + v^{a}(\Gamma) \mathcal{E}_{a}(\Gamma),
\end{equation*}
and refer to $\{ v^{i}(\Gamma) \}_{i=1}^{n}$ as the (\textit{$\Gamma$-dependent}) \textit{quasivelocities}.
For brevity, we shall often drop the $\Gamma$-dependence in what follows.
The advantage of this constraint-adapted Hamel basis is the following:
\begin{equation*}
  \xi \in \mathfrak{d}
  \iff
  v^{a} = 0,
\end{equation*}
where it is implied on the right-hand side that the equality holds for every $a \in \{1, \dots, r\}$.
Hence we may simply drop some of the coordinates to take the constraints into account.

Given a standard ($\Gamma$-independent) basis $\{ E_{i} \}_{i=1}^{n}$ for $\mathfrak{s}$, we may write $\{ \mathcal{E}_{i}(\Gamma) \}_{i=1}^{n}$ as  
\begin{equation}
  \label{eq:bases-s}
  \mathcal{E}_{j}(\Gamma) = \mathcal{E}^{i}_{j}(\Gamma) E_{i}
  \iff
  E_{i} = (\mathcal{E}^{-1})_{i}^{j}(\Gamma)\, \mathcal{E}_{j}(\Gamma),
\end{equation}
where we abuse the notation as follows: $\mathcal{E}$ is the $n \times n$ matrix whose columns are $\{ \mathcal{E}_{j} \}_{j=1}^{n}$ so that $\mathcal{E}^{i}_{j}$ denotes the $(i,j)$-entry of $\mathcal{E}$ as well as the $i$-th component of $\mathcal{E}_{j}$ with respect to the standard basis $\{ E_{i} \}_{i=1}^{n}$.

Suppose that we have the structure constants $\{ c_{ij}^{k} \}_{1\le i,j,k\le n}$ for $\mathfrak{s}$ with respect to the standard basis $\{ E_{i} \}_{i=1}^{n}$, i.e.,
\begin{equation*}
  [E_{i}, E_{j}] = c_{ij}^{k} E_{k}.
\end{equation*}
Then the structure constants with respect to $\{ \mathcal{E}_{j} \}_{j=1}^{n}$ are also $\Gamma$-dependent: 
\begin{equation}
  \label{eq:mathcalC}
  [\mathcal{E}_{i}(\Gamma), \mathcal{E}_{j}(\Gamma)] = \mathcal{C}_{ij}^{k}(\Gamma)\,\mathcal{E}_{k}(\Gamma)
  \quad\text{with}\quad
  \mathcal{C}_{ij}^{k}(\Gamma)
  \defeq \mathcal{E}_{i}^{l}(\Gamma)\, \mathcal{E}_{j}^{m}(\Gamma)\, c_{lm}^{p}\, (\mathcal{E}^{-1})^{k}_{p}(\Gamma).
\end{equation}


\subsection{Eliminating the Lagrange Multipliers}
Let us eliminate the Lagrange multipliers and find a concrete coordinate expression for \eqref{eq:NHEP1} in terms of the $\Gamma$-dependent quasivelocities.

Imposing the constraint $\xi \in \mathfrak{d}$ and using the $\Gamma$-dependent basis $\{ \mathcal{E}_{i} \}_{i=1}^{n}$ for $\mathfrak{s}$, we write
\begin{equation}
  \label{eq:mu_and_p}
  \mu \defeq \left. \fd{\ell}{\xi} \right|_{\rm c},
  \qquad
  p_{i} \defeq \ip{ \mu }{ \mathcal{E}_{i} },
\end{equation}
where the subscript $\left.(\,\cdot\,)\right|_{\rm c}$ indicates that the constraint $\xi \in \mathfrak{d}(\Gamma)$ is applied \textit{after} computing what is inside $(\,\cdot\,)$.
As a result, $\{ p_{i} \}_{i=1}^{n}$ are written in terms of $\{ v^{\alpha} \}_{\alpha=1}^{n-r}$.

Let $\{ \mathbf{e}_{i} \}_{i=1}^{\dim X}$ be a basis for $X$ and $\{ \mathbf{e}_{*}^{i} \}_{i=1}^{\dim X}$ be its dual basis for $X^{*}$, and write
\begin{equation*}
  \Gamma = \Gamma_{i}\, \mathbf{e}_{*}^{i},
  \qquad
  \xi \Gamma = \kappa^{k}_{ij}\, \xi^{j}\, \Gamma_{k}\, \mathbf{e}_{*}^{i}
\end{equation*}
using constants $\kappa^{k}_{ij}$ that are determined by the $\mathfrak{s}$-representation~\eqref{eq:s-rep_on_X*} on $X^{*}$.
Then we have:
\begin{theorem}
  \label{thm:NHEP-quasivelocities}
  The nonholonomic Euler--Poincar\'e equations~\eqref{eq:NHEP} are written in coordinates as
  \begin{subequations}
    \label{eq:NHEP-quasivel}
    \begin{align}
      \label{eq:NHEP1-quasivel}
      \dot{p}_{\alpha}
      &= -\mathcal{B}_{\alpha\beta}^{i}(\Gamma)\, p_{i}\, v^{\beta}
        + \mathcal{K}_{\alpha j}^{k} \left.\pd{\ell}{\Gamma_{j}}\right|_{\rm c} \Gamma_{k},
      \\
      \label{eq:NHEP2-quasivel}
      \dot{\Gamma}_{i} &= -\varkappa^{k}_{i\beta}(\Gamma)\, v^{\beta}\, \Gamma_{k},
    \end{align}
  \end{subequations}
  where
  \begin{gather}
    \mathcal{K}_{ij}^{k}
    \defeq \ip{ \mathbf{K}\parentheses{ \mathbf{e}_{j}, \mathbf{e}_{*}^{k} } }{ \mathcal{E}_{i} },
    \qquad
    \varkappa^{k}_{i\beta}(\Gamma)
    \defeq \kappa^{k}_{ij}\, \mathcal{E}_{\beta}^{j}(\Gamma),
    \\
    \label{eq:mathcalF}
    \mathcal{B}_{\alpha\beta}^{i}(\Gamma) \defeq \mathcal{C}_{\alpha\beta}^{i}(\Gamma) + \mathcal{F}_{\alpha\beta}^{i}(\Gamma)
    \quad\text{with}\quad
    \mathcal{F}^{i}_{\alpha\beta}(\Gamma)
    \defeq D^{j}\mathcal{E}_{\alpha}^{l}(\Gamma)\, (\mathcal{E}^{-1})_{l}^{i}(\Gamma)\, \varkappa^{k}_{j\beta}(\Gamma)\, \Gamma_{k},
  \end{gather}
  and $D^{j}\mathcal{E}_{\alpha}^{i}$ stands for $\tpd{\mathcal{E}_{\alpha}^{i}}{\Gamma_{j}}$.
\end{theorem}
\begin{remark}
  \label{rem:p-v}
  As explained above, $\{ p_{i} \}_{i=1}^{n}$ are written in terms of $\{ v^{\alpha} \}_{\alpha=1}^{n-r}$, and so \eqref{eq:NHEP-quasivel} gives the closed set of equations for $\{ v^{\alpha} \}_{\alpha=1}^{n-r}$ and $\Gamma$.
\end{remark}
\begin{proof}
  It is straightforward to see that \eqref{eq:NHEP2-quasivel} follows from \eqref{eq:NHEP2}:
  Since $\xi^{j} = v^{\beta}\, \mathcal{E}_{\beta}^{j}$ for every $\xi \in \mathfrak{d}$, using the definitions of $\kappa$ and $\varkappa$ from above,
  \begin{equation*}
    \dot{\Gamma}_{i}
    = -(\xi \Gamma)_{i}
    = -\kappa^{k}_{ij}\, \xi^{j}\, \Gamma_{k}
    = -\varkappa^{k}_{i\beta}(\Gamma)\, v^{\beta}\, \Gamma_{k}.
  \end{equation*}

  It remains to derive \eqref{eq:NHEP1-quasivel} from \eqref{eq:NHEP1}.
  Imposing the constraint $\xi \in \mathfrak{d}$ to \eqref{eq:NHEP1} and using $\mu$ defined in \eqref{eq:mu_and_p}, we have
  \begin{equation*}
    \dot{\mu} = \ad_{\xi}^{*} \mu
    + \mathbf{K}\parentheses{ \left. \fd{\ell}{\Gamma} \right|_{\rm c}, \Gamma}
    + \lambda_{a} \psi^{a}(\Gamma).
  \end{equation*}
  Then, taking the paring of both sides with $\mathcal{E}_{\alpha}$, we have
  \begin{equation*}
    \ip{ \dot{\mu} }{ \mathcal{E}_{\alpha} }
    = \ip{ \ad_{\xi}^{*} \mu }{ \mathcal{E}_{\alpha} }
    + \ip{ \mathbf{K}\parentheses{ \left. \fd{\ell}{\Gamma} \right|_{\rm c}, \Gamma} }{ \mathcal{E}_{\alpha} }
    + \underbrace{ \ip{ \lambda_{a} \psi^{a} }{ \mathcal{E}_{\alpha} } }_{0}.
  \end{equation*}
  The left-hand side becomes
  \begin{align*}
    \ip{ \dot{\mu} }{ \mathcal{E}_{\alpha} }
    &=  \od{}{t} \ip{ \mu }{ \mathcal{E}_{\alpha} }
      - \ip{ \mu }{ \dot{\mathcal{E}}_{\alpha} } \\
    &= \dot{p}_{\alpha} - \ip{ \mu }{ \dot{\mathcal{E}}_{\alpha} } \\
    &= \dot{p}_{\alpha} - \ip{ \mu }{ D\mathcal{E}_{\alpha}(\Gamma) \dot{\Gamma} } \\
    &= \dot{p}_{\alpha} + \ip{ \mu }{ D\mathcal{E}_{\alpha}(\Gamma) \xi \Gamma },
  \end{align*}
  but then
  \begin{align*}
    D\mathcal{E}_{\alpha}(\Gamma) \xi \Gamma
    &= \pd{\mathcal{E}_{\alpha}^{l}}{\Gamma_{j}} (\xi \Gamma)_{j} E_{l} \\
    &= D^{j}\mathcal{E}_{\alpha}^{l}(\Gamma)\, \varkappa^{k}_{j\beta}(\Gamma)\, v^{\beta}\, \Gamma_{k}\, (\mathcal{E}^{-1})_{l}^{i}(\Gamma)\, \mathcal{E}_{i}(\Gamma) \\
    &= \mathcal{F}^{i}_{\alpha\beta}(\Gamma)\, v^{\beta}\, \mathcal{E}_{i}(\Gamma).
  \end{align*}
  Therefore, using the definition of $p_{i}$ from \eqref{eq:mu_and_p}, we obtain
  \begin{equation*}
    \ip{ \dot{\mu} }{ \mathcal{E}_{\alpha} }
    = \dot{p}_{\alpha} + \mathcal{F}^{i}_{\alpha\beta}(\Gamma)\, p_{i}\, v^{\beta}.
  \end{equation*}
  On the other hand, a straightforward computation yields
  \begin{equation*}
    \ip{ \ad_{\xi}^{*} \mu }{ \mathcal{E}_{\alpha} }
    = \mathcal{C}_{\beta\alpha}^{i}\, p_{i}\, v^{\beta}
    = -\mathcal{C}_{\alpha\beta}^{i}\, p_{i}\, v^{\beta}.
  \end{equation*}
  Moreover, using the bases $\{ \mathbf{e}_{i} \}_{i=1}^{\dim X}$ for $X$ and $\{ \mathbf{e}_{*}^{i} \}_{i=1}^{\dim X}$ for $X^{*}$ introduced earlier, we have
  \begin{equation*}
    \mathbf{K}\parentheses{ \left.\fd{\ell}{\Gamma}\right|_{\rm c}, \Gamma}
    = \mathbf{K}\parentheses{
      \left.\pd{\ell}{\Gamma_{j}}\right|_{\rm c} \mathbf{e}_{j},
      \Gamma_{k}\, \mathbf{e}_{*}^{k}
      } \\
    = \mathbf{K}\parentheses{ \mathbf{e}_{j}, \mathbf{e}_{*}^{k} } \left.\pd{\ell}{\Gamma_{j}}\right|_{\rm c} \Gamma_{k},
  \end{equation*}
  and so
  \begin{equation*}
    \ip{ \mathbf{K}\parentheses{ \left.\fd{\ell}{\Gamma}\right|_{\rm c}, \Gamma} }{ \mathcal{E}_{\alpha} }
    = \ip{ \mathbf{K}\parentheses{ \mathbf{e}_{j}, \mathbf{e}_{*}^{k} } }{ \mathcal{E}_{\alpha} } \left.\pd{\ell}{\Gamma_{j}}\right|_{\rm c} \Gamma_{k}
    = \mathcal{K}_{\alpha j}^{k} \left.\pd{\ell}{\Gamma_{j}}\right|_{\rm c} \Gamma_{k}. \qedhere
  \end{equation*}
\end{proof}

\subsection{Lagrangian and Energy}
In what follows, we assume that the reduced Lagrangian $\ell$ from \eqref{eq:ell} takes the following ``kinetic minus potential'' form:
\begin{equation}
  \label{eq:ell-simple}
  \ell(\xi, \Gamma) = \frac{1}{2} \mathbb{G}_{ij}\,\xi^{i} \xi^{j} - U(\Gamma)
\end{equation}
with a constant $n \times n$ matrix $\mathbb{G}$ and $U\colon X^{*} \to \R$.
Since $\xi^{i} = \mathcal{E}_{j}^{i}(\Gamma)\, v^{j}(\Gamma)$, we have
\begin{equation}
  \label{eq:mathcalG_ij}
  \frac{1}{2} \mathbb{G}_{ij}\,\xi^{i} \xi^{j} = \frac{1}{2} \mathcal{G}_{ij}(\Gamma)\, v^{i}(\Gamma)\, v^{j}(\Gamma)
  \quad\text{with}\quad
  \mathcal{G}_{ij}(\Gamma) \defeq \mathbb{G}_{kl}\, \mathcal{E}_{i}^{k}(\Gamma)\, \mathcal{E}_{j}^{l}(\Gamma).
\end{equation}
Then we see that
\begin{align*}
  \pd{\ell}{\xi^{i}} = \mathbb{G}_{ij}\, \xi^{j} = \mathbb{G}_{ij}\, \mathcal{E}_{k}^{j}\, v^{k}
  &\implies
    \mu_{i} = \left.\pd{\ell}{\xi^{i}}\right|_{\rm c} = \mathbb{G}_{ij}\, \mathcal{E}_{\beta}^{j}\, v^{\beta} \\
  &\implies
    p_{i}
    = \mu_{k}\, \mathcal{E}_{i}^{k}
    = \mathbb{G}_{kj}\, \mathcal{E}_{i}^{k}\, \mathcal{E}_{\beta}^{j}\, v^{\beta}
    = \mathcal{G}_{i\beta}\, v^{\beta},
\end{align*}
giving the concrete relationship between $\{ p_{i} \}_{i=1}^{n}$ and $\{ v^{\alpha} \}_{\alpha=1}^{n-r}$ alluded in \Cref{rem:p-v}.
Then the constrained energy function
\begin{equation}
  \label{eq:mathscrE}
  \mathscr{E} \colon \R^{n-r} \times X^{*} \to \R;
  \qquad
  \mathscr{E}(v^{\alpha}, \Gamma)
  \defeq \frac{1}{2} \mathcal{G}_{\alpha\beta}(\Gamma)\, v^{\alpha} v^{\beta} + U(\Gamma)
\end{equation}
is an invariant of \eqref{eq:NHEP-quasivel} because this is the energy of the system expressed in the quasivelocities.

\section{Pendulum Skate}
\label{sec:PendulumSkate}
Let us now come back to our motivating example and apply \Cref{thm:NHEP-quasivelocities} to the pendulum skate.

\subsection{Equations of Motion}
The $\Gamma$-dependent Hamel basis for the pendulum skate is an extension of the hybrid frame introduced in \cite{GzPu2020}:
\begin{equation*}
  \se(3) \cong \R^{3} \times \R^{3} = \mathfrak{d}(\Gamma) \oplus \mathfrak{v}(\Gamma),
\end{equation*}
where
\begin{equation}
  \label{eq:mathcalE-PendulumSkate}
  \begin{split}
    \mathfrak{d} &= \Span\{
    \mathcal{E}_{1}  \defeq (\bE1, \bzero),\,
    \mathcal{E}_{2}  \defeq (\bGamma, \bzero),\,
    \mathcal{E}_{3}  \defeq (\bzero, \bE1)
    \}, \\
    \mathfrak{v} & \defeq \Span\{
    \mathcal{E}_{4}  \defeq (\bE1 \times \bGamma, \bzero),\,
    \mathcal{E}_{5}  \defeq (\bzero, \bGamma),\,
    \mathcal{E}_{6}  \defeq (\bzero, \bE1 \times \bGamma)
    \},
  \end{split}
\end{equation}
and we equip $\se(3) \cong \R^{6}$ with an inner product by using the dot product on $\R^{6}$.
Then, due to the pitch constancy condition~\eqref{eq:pitch_constancy}, these define orthonormal bases for $\mathfrak{d}$ and $\mathfrak{v}$ with $\mathfrak{v}$ being the orthogonal complement of $\mathfrak{d}$, and also they together form an orthonormal basis for $\se(3)$ as well.

Since the commutator in $\se(3) \cong \R^{3} \times \R^{3}$ given by
\begin{equation*}
  \ad_{(\bOmega_{1}, \bY_{1})} (\bOmega_{2}, \bY_{2})
  = [(\bOmega_{1}, \bY_{1}), (\bOmega_{2}, \bY_{2})]
  = \parentheses{
    \bOmega_{1} \times \bOmega_{2},\,
    \bOmega_{1} \times \bY_{2} - \bOmega_{2} \times \bY_{1}
  },
\end{equation*}
whereas \eqref{eq:mathcalE-PendulumSkate} yields
\begin{equation}
  \label{eq:DmathcalE}
  D\mathcal{E}_{2} =
  \begin{bmatrix}
    I \\
    0
  \end{bmatrix},
  \quad
  D\mathcal{E}_{1} = D\mathcal{E}_{3} = 0,
\end{equation}
the $\Gamma$-dependent structure constants $\mathcal{C}^{k}_{\alpha\beta}$ defined in \eqref{eq:mathcalC} are actually independent of $\Gamma$ here:
\begin{equation*}
  \mathcal{C}^{k}_{\alpha\beta}\, p_{k}
  = \begin{bmatrix}
    0 & -p_{4} & 0 \\
    p_{4} & 0 & p_{6} \\
    0 & -p_{6} & 0 
  \end{bmatrix}
  \quad
  \forall (p_{1}, \dots, p_{6}) \in \R^{6}.
\end{equation*}
Note that we do not need the full $6 \times 6$ matrix $\mathcal{C}^{k}_{ij}\, p_{k}$.

We may then write $\xi = (\bOmega, \bY) \in \se(3)$ in terms of quasivelocities $\{v_{i}\}_{i=1}^{6}$:
\begin{equation*}
  \bOmega = v^{1} \bE1 + v^{2} \bGamma + v^{4} (\bE1 \times \bGamma),
  \qquad
  \bY = v^{3} \bE1 + v^{5} \bGamma + v^{6} (\bE1 \times \bGamma),
\end{equation*}
where, by the orthonormality,
\begin{equation}
  \label{eq:quasivel-PendulumSkate}
  \begin{array}{c}
    v^{1} = \bOmega \cdot \bE1 = \Omega_{1},
    \qquad
    v^{2} = \bOmega \cdot \bGamma,
    \qquad
    v^{3} = \bY \cdot \bE1 = Y_{1},
    \medskip\\
    v^{4} = \bOmega \cdot (\bE1 \times \bGamma),
    \qquad
    v^{5} = \bY \cdot \bGamma,
    \qquad
    v^{6} = \bY \cdot (\bE1 \times \bGamma).
  \end{array}
\end{equation}
Then the constraint $\xi = (\bOmega, \bY) \in \mathfrak{d}(\bGamma)$ is equivalent to $v^{a} = 0$ with $a \in \{4, 5, 6 \}$.

On the other hand, the Lagrangian~\eqref{eq:ell-PendulumSkate} becomes
\begin{equation}
  \begin{split}
    \ell(\bOmega, \bY, \bGamma)
    &= \frac{1}{2} \bOmega^{T} \mathbb{I} \bOmega
    + \frac{m}{2}\norm{ \bY + l \bOmega \times \bE3 }^2
    - m \mathrm{g} l \bGamma^{T} \bE3 \\
    &= \frac{1}{2} \bOmega^{T} \bar{\mathbb{I}} \bOmega
    + m l \bOmega \cdot (\bE3 \times \bY)
    + \frac{m}{2} \norm{\bY}^{2}
    - m \mathrm{g} l \Gamma_{3},
  \end{split}
\end{equation}
where
\begin{equation*}
  \bar{\mathbb{I}} \defeq \diag(\bar{I}_{1}, \bar{I}_{2}, I_{3})
  \quad\text{with}\quad
  \bar{I}_{i} \defeq I_{i} + m l^{2}
  \quad\text{for}\quad
  i = 1, 2.
\end{equation*}
Then
\begin{equation*}
  \pd{\ell}{\xi}
  = \parentheses{
    \pd{\ell}{\bOmega},
    \pd{\ell}{\bY}
  },
  \qquad
  \pd{\ell}{\bOmega} = \bar{\mathbb{I}} \bOmega + m l (\bE3 \times \bY),
  \qquad
  \pd{\ell}{\bY} = m l (\bOmega \times \bE3) + m \bY,
\end{equation*}
and thus we have
\begin{equation*}
  \mu = \left. \fd{\ell}{\xi} \right|_{\rm c}
  = \begin{bmatrix}
    \bar{I}_{1} v^{1} \\
    \bar{I}_{2} \Gamma_{2} v^{2} + m l v^{3} \\
    I_{3} \Gamma_{3} v^{2} \\
    m (l \Gamma_{2} v^{2} + v^{3}) \\
    -m l v^{1} \\
    0
  \end{bmatrix},
\end{equation*}
and
\begin{equation*}
  \begin{bmatrix}
    p_{1} \\
    p_{2} \\
    p_{3}
  \end{bmatrix}
  =
  \begin{bmatrix}
    \bar{I}_{1} v^{1} \\
    (\bar{I}_{2} \Gamma_{2}^{2} + I_{3} \Gamma_{3}^{2}) v^{2} + m l \Gamma_{2} v^{3} \\
    m (l \Gamma_{2} v^{2} + v^{3})
  \end{bmatrix},
  \qquad
  \begin{bmatrix}
    p_{4} \\
    p_{5} \\
    p_{6}
  \end{bmatrix}
  =
  \begin{bmatrix}
    (I_{3} - \bar{I}_{2}) \Gamma_{2} \Gamma_{3} v^{2} - m l \Gamma_{3} v^{3} \\
    -m l \Gamma_{2} v^{1} \\
    m l \Gamma_{3} v^{1}
  \end{bmatrix}.
\end{equation*}

Let us find the right-hand side of \eqref{eq:NHEP1-quasivel}.
We find
\begin{equation*}
  \mathcal{C}_{\alpha\beta}^{i}\, p_{i}\, v^{\beta}
  = \begin{bmatrix}
    -p_{4} v^{2} \\
    p_{4} v^{1} + p_{6} v^{3} \\
    -p_{6} v^{2}
  \end{bmatrix},
  \qquad
  \xi\Gamma = \bOmega \times \bGamma
  = v^{1}
  \begin{bmatrix}
    0 \\
    \Gamma_{3} \\
    -\Gamma_{2}
  \end{bmatrix}
\end{equation*}
since $\Gamma_{1} = 0$.
Hence we have, using \eqref{eq:DmathcalE},
\begin{equation*}
  \mathcal{F}^{i}_{\alpha\beta}(\Gamma)\, p_{i}\, v^{\beta}
  = \ip{ \mu }{ D\mathcal{E}_{\alpha}(\Gamma)\, \xi \Gamma } 
  = \ip{ \mu }{ D^{j}\mathcal{E}_{\alpha}(\Gamma)\, (\xi \Gamma)_{j} } 
  = \begin{bmatrix}
    0 \\
    (\Gamma_{3}\, p_{2} - \Gamma_{2}\, p_{3}) v^{1}
    0
  \end{bmatrix}.
\end{equation*}
Therefore, we obtain
\begin{align*}
  \mathcal{B}_{\alpha\beta}^{i}\, p_{i}\, v^{\beta}
  &= \parentheses{
  \mathcal{C}_{\alpha\beta}^{i} + \mathcal{F}^{i}_{\alpha\beta}
  } p_{i}\, v^{\beta} \\
  &= \begin{bmatrix}
    -p_{4} v^{2} \\
    p_{4} v^{1} + p_{6} v^{3} + (\Gamma_{3}\, p_{2} - \Gamma_{2}\, p_{3}) v^{1} \\
    -p_{6} v^{2}
  \end{bmatrix}
  = \begin{bmatrix}
    (\bar{I}_{2} - I_{3}) \Gamma_{2} \Gamma_{3} (v^{2})^{2} + m l \Gamma_{3} v^{2} v^{3} \\
    m l \Gamma_{3} v^{3} v^{1} \\
    -m l \Gamma_{3} v^{1} v^{2}
  \end{bmatrix}.
\end{align*}
Using \eqref{eq:K-PendulumSkate}, we also have
\begin{equation*}
  \mathcal{K}_{\alpha j}^{k} \left.\pd{\ell}{\Gamma_{j}}\right|_{\rm c} \Gamma_{k}
  = m \mathrm{g} l
  \begin{bmatrix}
    \Gamma_{2} \\
    0 \\
    0
  \end{bmatrix}.
\end{equation*}

As a result, the nonholonomic Euler--Poincar\'e equations~\eqref{eq:NHEP-quasivel} become
\begin{subequations}
  \label{eq:NHEP-quasivel-PendulumSkate}
  \begin{align}
    \begin{bmatrix}
      \dot{p}_{1} \\
      \dot{p}_{2} \\
      \dot{p}_{3}
    \end{bmatrix}
    &= \od{}{t}
    \begin{bmatrix}
      \bar{I}_{1} v^{1} \\
      (\bar{I}_{2} \Gamma_{2}^{2} + I_{3} \Gamma_{3}^{2}) v^{2} + m l \Gamma_{2} v^{3} \\
      m (l \Gamma_{2} v^{2} + v^{3})
    \end{bmatrix}
    \nonumber\\
    &= \begin{bmatrix}
      (\bar{I}_{2} - I_{3}) \Gamma_{2} \Gamma_{3} (v^{2})^{2} + m l \Gamma_{3} v^{2} v^{3} + m \mathrm{g} l \Gamma_{2} \\
      m l \Gamma_{3} v^{3} v^{1} \\
      -m l \Gamma_{3} v^{1} v^{2}
    \end{bmatrix}
    \label{eq:NHEP1-PendulumSkate}
  \end{align}
  coupled with
  \begin{equation}
    \label{eq:NHEP2-PendulumSkate}
    \begin{bmatrix}
      \dot{\Gamma}_{2} \\
      \dot{\Gamma}_{3}
    \end{bmatrix}
    = v^{1}
    \begin{bmatrix}
      \Gamma_{3} \\
      -\Gamma_{2}
    \end{bmatrix}.
  \end{equation}
\end{subequations}

\subsection{Invariants}
One can see by inspection that \eqref{eq:NHEP-quasivel-PendulumSkate} implies
\begin{equation*}
  \od{}{t} \parentheses{ p_{2} - l \Gamma_{2} p_{3} } = 0,
\end{equation*}
which shows that
\begin{equation}
  \label{eq:C1-PendulumSkate}
  C_{1} \defeq p_{2} - l \Gamma_{2} p_{3} = (I_{2} \Gamma_{2}^{2} + I_{3} \Gamma_{3}^{2}) v^{2}
\end{equation}
is an invariant of the system---called $J_{1}$ in \cite[Eq.~(20)]{GzPu2020}.

We also notice by inspection that
\begin{equation*}
  \dot{p}_{3} = - \frac{ m l C_{1} }{ I_{2} \Gamma_{2}^{2} + I_{3} \Gamma_{3}^{2} } \dot{\Gamma}_{2}.
\end{equation*}
However, since $\Gamma_{2}^{2} + \Gamma_{3}^{2} = 1$ (due to the pitch constancy $\Gamma_{1} = 0$ from \eqref{eq:pitch_constancy}), we have
\begin{equation*}
  \dot{p}_{3}
  = - \frac{ m l C_{1} }{ (I_{2} - I_{3}) \Gamma_{2}^{2} + I_{3} } \dot{\Gamma}_{2}
  \implies
  \od{p_{3}}{\Gamma_{2}}
  = - \frac{ m l C_{1} }{ (I_{2} - I_{3}) \Gamma_{2}^{2} + I_{3} }.
\end{equation*}
Therefore,
\begin{equation*}
  p_{3} = - m l C_{1} \int \frac{ 1}{ (I_{2} - I_{3}) \Gamma_{2}^{2} + I_{3} }\,d\Gamma_{2},
\end{equation*}
implying that
\begin{align}
  C_{2}
  &\defeq \frac{p_{3}}{m} + l C_{1} \int \frac{1}{ (I_{2} - I_{3}) \Gamma_{2}^{2} + I_{3} }\,d\Gamma_{2} \nonumber\\
  &= l \Gamma_{2} v^{2} + v^{3} + \frac{l}{I_{3}(I_{2} - I_{3})} C_{1} \arctan\parentheses{ \sqrt{ \frac{I_{2} - I_{3}}{I_{3}} }\,\Gamma_{2} }
    \label{eq:C2-PendulumSkate}
\end{align}
is also an invariant of the system---called $J_{2}$ in \cite[Eq.~(21)]{GzPu2020}, where we assumed $I_{2} > I_{3}$ (and shall do so for the rest of the paper) because this is the case with realistic skaters as mentioned in \cite[Proof of Theorem~2]{GzPu2020}.

\subsection{Equilibria}
We shall use the following shorthands in what follows:
\begin{equation*}
  z \defeq (\bOmega, \bY, \bGamma),
  \qquad
  \zeta \defeq \left( v^{1}, v^{2}, v^{3}, \Gamma_{2}, \Gamma_{3} \right).
\end{equation*}
Note that $z$ denotes the original dependent variables in the nonholonomic Euler--Poincar\'e equations~\eqref{eq:NHEP-PendulumSkate} with Lagrange multipliers, whereas $\zeta$ denotes those in \eqref{eq:NHEP-quasivel-PendulumSkate} using quasivelocities.
Now, let us rewrite the system~\eqref{eq:NHEP-quasivel-PendulumSkate} as
\begin{multline}
  \label{eq:vector_field-PendulumSkate}
  \dot{\zeta} = f(\zeta)
  \quad\text{with}\quad
  f(\zeta) \defeq \Biggl(
  \frac{
    \bigl(
    \left(\bar{I}_{2}-I_{3}\right) \Gamma_{3} (v^{2})^{2} + m \mathrm{g} l
    \bigr) \Gamma_{2}
    + m l \Gamma_{3} v^{2} v^{3}
  }{ \bar{I}_{1} }, \\
  \frac{2 (I_{3}-I_{2}) \Gamma_{2} \Gamma_{3} v^{1} v^{2}}{I_{2}\Gamma_{2}^2 + I_{3} \Gamma_{3}^2},\,
  -\frac{2 l I_{3} \Gamma_{3} v^{1} v^{2}}{I_{2}\Gamma_{2}^2 + I_{3} \Gamma_{3}^2},\,
  v^{1} \Gamma_{3},\,
  -v^{1} \Gamma_{2}
 \Biggr).
\end{multline}
Then one finds that the equilibria are characterized as follows:
\begin{equation}
  \label{eq:equilibria-quasivel}
  v^{1} = 0
  \quad\text{and}\quad
  \parentheses{ (\bar{I}_{2} - I_{3}) \Gamma_{3} (v^{2})^{2} + m \mathrm{g} l } \Gamma_{2}
  + m l \Gamma_{3} v^{2} v^{3} = 0.
\end{equation}
Note that, in view of \eqref{eq:quasivel-PendulumSkate}, $v^{1} = 0$ is equivalent to $\Omega_{1} = 0$.

Let us impose the upright position, i.e., $\Gamma_{3} = 1$ or equivalently $\bGamma = \bE3$.
Then the constraints $v^{a} = 0$ for $a = 4, 5, 6$ yield $\Omega_{2} = Y_{2} = Y_{3} = 0$, i.e., $\bOmega = \Omega_{0} \bE3$ and $\bY = Y_{0} \bE1$ for arbitrary $\Omega_{0}, Y_{0} \in \R$.
Furthermore, the second equation in \eqref{eq:equilibria-quasivel} reduces to $v^{2} v^{3} = 0$.
If $v^{2} = 0$ then $\Omega_{3} = 0$ and thus we have the \textit{sliding equilibrium}
\begin{equation}
  \label{eq:zeta_sl}
  z_{\rm sl} \defeq (\bzero, Y_{0}\bE1, \bE3)
  \quad\text{or}\quad
  \zeta_{\rm sl} \defeq (0, 0, Y_{0}, 0, 1),
\end{equation}
whereas $v^{3} = 0$ gives $Y_{1} = 0$ and gives the \textit{spinning equilibrium}
\begin{equation}
  \label{eq:zeta_sp}
  z_{\rm sp} \defeq (\Omega_{0}\bE3, \bzero, \bE3)
  \quad\text{or}\quad
  \zeta_{\rm sp} \defeq (0, \Omega_{0}, 0, 0, 1).
\end{equation}

As mentioned in the Introduction, these are equilibria of the reduced system and hence are relative equilibria of the original equations of motion, and stability here refers to the Lyapunov stability of the equilibrium of the reduced system.

\subsection{Stability of Equilibria}
\label{ssec:stability}
Let us first discuss the stability of the spinning equilibrium:
\begin{proposition}[Stability of spinning equilibrium]
  Suppose that $I_{3} < I_{2}$.
  \label{prop:spinning_eq}
  \leavevmode
  \begin{enumerate}[(i)]
  \item If $I_{3} + m l^{2} < I_{2}$, then the spinning equilibrium~\eqref{eq:zeta_sp} is unstable.
  \item If
    \begin{equation}
      \label{eq:zeta_sp-stability_condition}
      I_{3} +  m l^{2} > I_{2}
      \quad\text{and}\quad
      \abs{\Omega_{0}} > \sqrt{ \frac{m \mathrm{g} l}{ I_{3} + m l^{2} - I_{2} } },
    \end{equation}
    then the spinning equilibrium~\eqref{eq:zeta_sp} is stable.
  \end{enumerate}
\end{proposition}
\begin{proof}
  See \Cref{proof:spinning_eq}.
\end{proof}

On the other hand, the sliding equilibrium is always unstable:
\begin{proposition}[Stability of sliding equilibrium]
  \label{prop:sliding_eq}
  The sliding equilibrium~\eqref{eq:zeta_sl} is linearly unstable.
\end{proposition}
\begin{proof}
  The Jacobian $Df$ of the vector field $f$ from \eqref{eq:vector_field-PendulumSkate} at the sliding equilibrium~\eqref{eq:zeta_sl} is
  \begin{equation*}
    Df(\zeta_{\rm sl}) =
    \begin{bmatrix}
      0 & \frac{m l Y_{0}}{I_{1} + m l^{2}} & 0 & \frac{m \mathrm{g} l}{I_{1} + m l^{2}} \\
      0 & 0 & 0 & 0 \\
      0 & 0 & 0 & 0 \\
      1 & 0 & 0 & 0 \\
    \end{bmatrix},
  \end{equation*}
  and its eigenvalues are
  \begin{equation*}
    \left\{0, 0, 0, \pm\sqrt{ \frac{m \mathrm{g} l}{I_{1} + m l^{2}} } \right\},  
  \end{equation*}
  The presence of a positive eigenvalue implies the assertion by the Instability from Linearization criterion~(see, e.g., \citet[p.216]{Sastry1999}).
\end{proof}

\section{Controlled Lagrangian and Matching}
\label{sec:ControlledLagrangian}
The goal of this section is to apply the method of Controlled Lagrangians to the nonholonomic Euler--Poincar\'e equations~\eqref{eq:NHEP-quasivel}.
Our formulation using the $\Gamma$-dependent quasivelocities helps us extend the method of Controlled Lagrangians of \citet{BlLeMa2000} to our system~\eqref{eq:NHEP-quasivel}.
Indeed, the arguments to follow in this section almost exactly parallel those of \cite[Section~2]{BlLeMa2000}.

\subsection{Method of Controlled Lagrangians}
The basic idea of the method of Controlled Lagrangians (CL) is to find a feedback control law for a mechanical system that does not destroy the underlying structure of the mechanical system.
The equations of motion for a mechanical system often possess invariants, such as the energy, Noether invariants, and Casimirs.
However, with controls applied to the system, these invariants typically do not remain invariant.
As a remedy for such a drawback, the method of CL seeks those controls that preserve such invariants by ``matching'' the controlled system with the same form of (uncontrolled) equations of motion with a \textit{different} Lagrangian.
As a result, the controlled system retains those invariants possessed by the original uncontrolled system, albeit in slightly different expressions.
In the next two subsections to follow, we shall explain in detail how the matching works for the pendulum skate.

The method of CL is particularly useful for stabilizing equilibria of mechanical systems, as one may employ a combination of the invariants as a Lyapunov function.
In other words, the main advantage of the method of CL is that the energy--Casimir/momentum method (see, e.g., \citet{Ar1966b} and \citet[Section~1.7]{MaRa1999}) is applicable to the controlled system.

\subsection{Controlled Lagrangian}
Our motivating example is the stabilization of the sliding equilibria of the pendulum skate---shown to to be unstable in \Cref{prop:sliding_eq}---using an internal wheel; see \Cref{fig:PendulumSkate-Controlled}.

\begin{figure}[hbtp]
  \centering
  \includegraphics[width=0.25\linewidth]{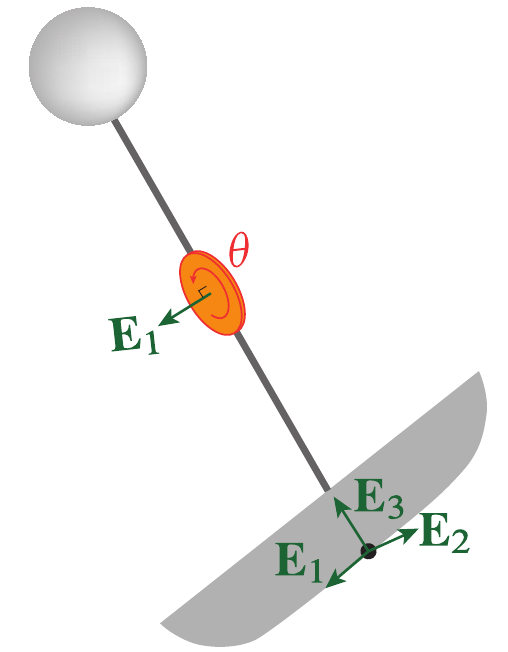}
  \caption{Pendulum skate controlled by an internal rotor attached to the center of mass. Its axis of rotation is aligned with $\bE1$, i.e., the edge of the skate.}
  \label{fig:PendulumSkate-Controlled}
\end{figure}

Following \cite{BlLeMa2000}, let $\mathsf{H}$ be an $s$-dimensional Abelian Lie group and $\mathfrak{h} \cong \R^{s}$ be its Lie algebra; practically $\mathsf{H}$ gives the configuration space of $s$ internal rotors, i.e., $\mathsf{H} = \mathbb{T}^{s}$.
We shall replace the reduced Lagrangian~\eqref{eq:ell-simple} by
\begin{equation}
  \label{eq:ell_r}
  \ell_{\rm r}\colon \mathfrak{s} \times \mathfrak{h} \times X^{*} \to \R;
  \qquad
  \ell_{\rm r}\parentheses{ \xi, \dot{\theta}^{a}, \Gamma }
     \defeq K\parentheses{ \xi, \dot{\theta}^{a} } - U(\Gamma),
\end{equation}
where
\begin{equation*}
  K\parentheses{ \xi, \dot{\theta}^{a} }
     \defeq \frac{1}{2} \mathbb{G}_{ij}\,\xi^{i} \xi^{j}
     + \mathbb{G}_{ia}\, \xi^{i} \dot{\theta}^{a}
     + \frac{1}{2} \mathbb{G}_{ab}\, \dot{\theta}^{a} \dot{\theta}^{b}
\end{equation*}
with a constant symmetric kinetic energy tensor $\mathbb{G}$, i.e., $\mathbb{G}_{ba} = \mathbb{G}_{ab}$; also $\mathbb{G}_{ai}$ and $\mathbb{G}_{ia}$, seen as matrices, are transposes to each other.

Then the equations of motion with control inputs (torques) $\{ u_{a} \}_{a=1}^{s}$ applied to the $\mathsf{H}$-part (internal rotors) are 
\begin{subequations}
  \label{eq:cNHEP}
  \begin{align}
    \label{eq:cNHEP1}
    \dot{p}_{\alpha} &= -\mathcal{B}_{\alpha\beta}^{i}(\Gamma)\, p_{i}\, v^{\beta}
                       - \mathcal{K}_{\alpha j}^{k}\, \pd{U}{\Gamma_{j}} \Gamma_{k},
    \\
    \label{eq:cNHEP2}
    \dot{\pi}_{a} &= u_{a},
    \\
    \dot{\Gamma}_{i} &= -\varkappa^{k}_{i\beta}(\Gamma)\, v^{\beta}\, \Gamma_{k},
  \end{align}
\end{subequations}
where $\{ p_{i} \}_{i=1}^{n}$ and $\{ \pi_{a} \}_{a=1}^{s}$ are (re-)defined as follows using the modified Lagrangian~\eqref{eq:ell_r}:
\begin{equation}
  \label{eq:p_and_pi}
  p_{i} \defeq \ip{ \left. \fd{\ell_{\rm r}}{\xi} \right|_{\rm c} }{ \mathcal{E}_{i} }
  = \mathcal{G}_{i\beta}\, v^{\beta}
  + \mathcal{G}_{ia}\, \dot{\theta}^{a},
  \qquad
  \pi_{a} \defeq \left. \pd{\ell}{\dot{\theta}^{a}} \right|_{\rm c}
  = \mathcal{G}_{a\alpha}\, v^{\alpha} + \mathbb{G}_{ab}\, \dot{\theta}^{b},
\end{equation}
where
\begin{equation}
  \label{eq:mathcalG_ia}
  \mathcal{G}_{ia}(\Gamma) \defeq \mathcal{E}_{i}^{j}(\Gamma)\, \mathbb{G}_{ja},
\end{equation}
and $\mathcal{G}_{ai}$ is the transpose of $\mathcal{G}_{ia}$.
Notice that $\mathcal{G}_{ia}$ is slightly different from $\mathcal{G}_{ij}$ defined in \eqref{eq:mathcalG_ij} and should be distinguished based on the type of indices just like the $\mathbb{G}$'s defined above.

As explained in the previous subsection, we would like to ``match'' the control system~\eqref{eq:cNHEP} with an uncontrolled equations of motion with a different Lagrangian.
Specifically, let
\begin{equation*}
  \tilde{\ell}_{\rm r}\colon \mathfrak{s} \times \mathfrak{h} \times X^{*} \to \R;
  \qquad
  \tilde{\ell}_{\rm r}\parentheses{ \xi, \dot{\theta}^{a}, \Gamma }
  \defeq \tilde{K}\parentheses{ \xi, \dot{\theta}^{a} } - U(\Gamma)
\end{equation*}
be a reduced Lagrangian with a modified kinetic energy $\tilde{K}$ (arbitrary form for now), and consider the corresponding nonholonomic Euler--Poincar\'e equation as in \eqref{eq:cNHEP} but \textit{without control input} $u_{a}$:
\begin{subequations}
  \label{eq:NHEP-tildel}
  \begin{align}
    \label{eq:NHEP1-tildel}
    \dot{\tilde{p}}_{\alpha}
    &= -\mathcal{B}_{\alpha\beta}^{i}(\Gamma)\, v^{\beta} \tilde{p}_{i}
      - \mathcal{K}_{\alpha j}^{k}\, \pd{U}{\Gamma_{j}} \Gamma_{k},
    \\
    \label{eq:NHEP2-tildel}
    \dot{\tilde{\pi}}_{a} &= 0,
    \\
    \dot{\Gamma}_{i} &= -\varkappa^{k}_{i\beta}(\Gamma)\, v^{\beta}\, \Gamma_{k},
  \end{align}
\end{subequations}
where
\begin{equation*}
  \tilde{p}_{i}\defeq \ip{ \left. \fd{\tilde{\ell}_{\rm r}}{\xi} \right|_{\rm c} }{ \mathcal{E}_{i} },
  \qquad
  \tilde{\pi}_{a} \defeq \left. \pd{\tilde{\ell}_{\rm r}}{\dot{\theta}^{a}} \right|_{\rm c}.
\end{equation*}
Our goal is to find a modified kinetic energy $\tilde{K}$ so that \eqref{eq:NHEP-tildel} is equivalent to the controlled system~\eqref{eq:cNHEP}.
To this end, we follow \cite{BlLeMa2000} and consider the following from of the modified kinetic energy:
\begin{equation}
  \label{eq:tilde_ell_r}
  \begin{split}
    \tilde{K}\parentheses{ \xi, \dot{\theta}^{a} }
    &\defeq K\parentheses{ \xi^{i}, \dot{\theta}^{a} + \tau^{a}_{i}\, \xi^{i} }
      + \frac{1}{2} \sigma_{ab}\, \tau^{a}_{i} \tau^{b}_{j}\, \xi^{i} \xi^{j} \\
    &\quad + \frac{1}{2} (\rho_{ab} - \mathbb{G}_{ab})
      \Bigl( \dot{\theta}^{a} + (\mathbb{G}^{ac}\mathbb{G}_{ci} + \tau^{a}_{i}) \xi^{i} \Bigr)
      \Bigl( \dot{\theta}^{b} + (\mathbb{G}^{bc}\mathbb{G}_{cj} + \tau^{b}_{j}) \xi^{j} \Bigr),
  \end{split}
\end{equation}
where $\tau^{a}_{i}$, $\sigma_{ab}$, and $\rho_{ab}$ are all constant matrices and the last two are symmetric, and are to be determined to realize the matching.
Then we have
\begin{equation}
  \label{eq:tilde_p}
  \begin{split}
    \tilde{p}_{i} &= \mathcal{G}_{i\alpha}\,v^{\alpha} + \mathcal{G}_{ia} \dot{\theta}^{a}
    + \parentheses{ \mathcal{G}_{ib} + \sigma_{ab}\, \mathcal{T}^{a}_{i} } \mathcal{T}^{b}_{\alpha}\, v^{\alpha}
    + \mathbb{G}_{ab}\,\rho^{bc}\, \tilde{\pi}_{c}\, \mathcal{T}^{a}_{i} \\
    &\quad + (\rho_{ab} - \mathbb{G}_{ab}) \rho^{ad}\, \tilde{\pi}_{d} \parentheses{ \mathbb{G}^{bc} \mathcal{G}_{ci} + \mathcal{T}^{b}_{i} }
  \end{split}
\end{equation}
and
\begin{equation}
  \label{eq:tilde_pi}
  \tilde{\pi}_{a} = \left. \rho_{ab} \Bigl( \dot{\theta}^{b} + \mathbb{G}^{bc} \mathbb{G}_{ci}\, \xi^{i}+ \tau^{b}_{i}\, \xi^{i} \Bigr) \right|_{\rm c}
  = \rho_{ab} \Bigl( \dot{\theta}^{b} + \mathbb{G}^{bc} \mathcal{G}_{c\alpha}\, v^{\alpha}+ \mathcal{T}^{b}_{\alpha}\, v^{\alpha} \Bigr).
\end{equation}
See \Cref{proof:tilde_p} for the derivation of the expression for $\tilde{p}_{i}$, where we defined
\begin{equation}
  \label{eq:mathcalT}
  \mathcal{T}^{a}_{i} \defeq \tau^{a}_{j}\, \mathcal{E}_{i}^{j}.
\end{equation}

\subsection{Matching and Control Law}
It turns out that the same sufficient condition for matching from \cite[Theorem~2.1]{BlLeMa2000} works here:
\begin{proposition}
  \label{prop:matching}
  The controlled Euler--Poincar\'e equations \eqref{eq:cNHEP} and the Euler--Poincar\'e equations \eqref{eq:NHEP-tildel} with the (reduced) controlled Lagrangian $\tilde{\ell}_{\rm r}$ coincide if
  \begin{equation}
    \label{eq:matching}
    \tau^{a}_{i} = - \sigma^{ab} \mathbb{G}_{bi}
    \quad\text{and}\quad
    \sigma^{ab} + \rho^{ab} = \mathbb{G}^{ab}.
  \end{equation}
  Then the resulting control law is
  \begin{equation}
    \label{eq:u}
    u_{a} = \mathbb{G}_{ab} \parentheses{
          \mathcal{T}^{b}_{i}\, \mathcal{F}^{i}_{\alpha\beta}\, v^{\alpha} v^{\beta}
          - \mathcal{T}^{b}_{\alpha}\, \dot{v}^{\alpha}
          }.
  \end{equation}
\end{proposition}
\begin{proof}
  See \Cref{proof:matching}.
\end{proof}

\begin{remark}
  One can eliminate the acceleration $\dot{v}^{\alpha}$ from the control law~\eqref{eq:u} using \eqref{eq:cNHEP1}.
\end{remark}

\begin{example}[Pendulum skate with a rotor]
  \label{ex:PendulumSkate_with_rotor}
  Going back to the motivating example from \Cref{fig:PendulumSkate-Controlled}, we may write the reduced Lagrangian with a rotor as follows:
  \begin{equation}
    \label{eq:ell_r-PendulumSkate}
    \begin{split}
      \ell_{\rm r}&\colon \se(3) \times T_{1}\mathbb{S}^{1} \times (\R^{3})^{*} \cong \R^{6} \times \R \times \R^{3} \to \R; \\
      \ell_{\rm r}&\parentheses{ \bOmega, \bY, \dot{\theta}, \bGamma } \defeq
      K_{\rm r}\parentheses{ \bOmega, \bY, \dot{\theta} }
      - m \mathrm{g} l \bGamma^{T} \bE3
    \end{split}
  \end{equation}
  with
  \begin{equation}
    \label{eq:K_r}
    \begin{split}
      K_{\rm r}\parentheses{ \bOmega, \bY, \dot{\theta} }
      &\defeq \frac{1}{2}\Bigl(
      (I_{1} + m l^{2})\Omega_{1}^{2}
      + J_{1}\bigl(\Omega_{1} + \dot{\theta} \bigr)^{2}
      + (\mathcal{I}_{2} + m l^{2}) \Omega_{2}^{2}
      + \mathcal{I}_{3} \Omega_{3}^{2}
      \Bigr) \\
      &\qquad + m l (\Omega_{2} Y_{1} - \Omega_{1} Y_{2})
      + m \norm{\bY}^{2}
      \\
      &= \frac{1}{2}\Bigl(
      (\mathcal{I}_{1} + m l^{2}) \Omega_{1}^{2}
      + (\mathcal{I}_{2} + m l^{2}) \Omega_{2}^{2}
      + \mathcal{I}_{3} \Omega_{3}^{2}
      + m l (\Omega_{2} Y_{1} - \Omega_{1} Y_{2})
      + m \norm{\bY}^{2}
      \Bigr) \\
      &\qquad + J_{1} \Omega_{1} \dot{\theta}
      + \frac{1}{2} J_{1} \dot{\theta}^{2},
    \end{split}
  \end{equation}
  where
  \begin{equation*}
    \mathcal{I}_{i} \defeq I_{i} + J_{i}
  \end{equation*}
  with $J_{i}$ ($i=1,2,3$) being the moments of inertia of the rotor; note that $m$ now denotes the total mass of the system \textit{including} the rotor.
  Hence we have
  \begin{equation*}
    \mathbb{G}_{ij} =
    \begin{bmatrix}
      \mathcal{I}_{1} + m l^{2} & 0 & 0 & 0 & -m l & 0 \\
      0 & \mathcal{I}_{2} + m l^{2} & 0 & m l & 0 & 0 \\
      0 & 0 & \mathcal{I}_{3} & 0 & 0 & 0 \\
      0 & m l & 0 & m & 0 & 0 \\
      -m l & 0 & 0 & 0 & m & 0 \\
      0 & 0 & 0 & 0 & 0 & m \\
    \end{bmatrix},
    \qquad
    \mathbb{G}_{ia} =
    \begin{bmatrix}
      J_{1} \\
      0 \\
      0 \\
      0 \\
      0 \\
      0
    \end{bmatrix}
    = J_{1} \be1,
    \qquad
    \mathbb{G}_{ab} = J_{1}.
  \end{equation*}
  Note that $\sigma_{ab} \eqdef \sigma$ and $\rho_{ab} \eqdef \rho$ are all scalars because there is only one internal rotor ($s = 1$).
  Then the matching conditions~\eqref{eq:matching} give
  \begin{align*}
    \tau^{a}_{i} = -\frac{J_{1}}{\sigma}\,\be1^{T}
    \iff
    \mathcal{T}^{a}_{i} = \tau^{a}_{j}\, \mathcal{E}_{i}^{j}
    &= -\frac{J_{1}}{\sigma}
    \begin{bmatrix}
      \be1^{T}\mathcal{E}_{1} & \cdots & \be1^{T}\mathcal{E}_{6}
    \end{bmatrix} \\
    &= -\frac{J_{1}}{\sigma}
    \begin{bmatrix}
      1 & \Gamma_{1} & 0 & 0 & 0 & 0
    \end{bmatrix} \\
    &= -\frac{J_{1}}{\sigma} \be1^{T},
  \end{align*}
  noting that $\Gamma_{1} = 0$, and
  \begin{equation*}
    \frac{1}{\sigma} + \frac{1}{\rho} = \frac{1}{J_{1}} \iff
    \rho = \frac{J_{1}}{1 - J_{1}/\sigma}.
  \end{equation*}
  Then \eqref{eq:u} gives the control law
  \begin{equation}
    \label{eq:u-PendulumSkate}
      u_{a} = \frac{J_{1}^{2}}{\sigma}\,\dot{v}^{1} 
      = \frac{J_{1}}{\sigma}\,\frac{\Gamma_{2} \bigl(m \mathrm{g} l + \left(m l^2 + I_{2}-I_{3}\right) \Gamma_{3} (v^{2})^2 \bigr)
    + m l \Gamma_{3} v^{2} v^{3}}{(I_{1} + m l^{2})/J_{1}},
  \end{equation}
  which is what we obtained in \cite[Eq.~(15d)]{GaOh2022} (with $\sigma = -J_{1}/\nu$) using a simpler and ad-hoc controlled Lagrangian.

  Let us find the controlled Lagrangian.
  First, again using $\Gamma_{1} = 0$,
  \begin{equation*}
    \mathcal{G}_{ai} = \mathbb{G}_{aj} \mathcal{E}_{i}^{j}\,
    = J_{1} \be1^{T} \mathcal{E}_{i}^{T}
    = J_{1}
    \begin{bmatrix}
      1 & \Gamma_{1} & 0 & \dots & 0
    \end{bmatrix}
    = J_{1} \be1^{T}.
  \end{equation*}
  Then \eqref{eq:tilde_pi} gives
   \begin{equation*}
    \tilde{\pi}_{a} = J_{1} v^{1} + \frac{J_{1}}{ 1 - J_{1}/\sigma } \dot{\theta}.
  \end{equation*}
  Since $\tilde{\pi}_{a}$ is conserved according to \eqref{eq:NHEP2-tildel}, we impose that
  \begin{equation*}
    \tilde{\pi}_{a} = 0
    \iff
    \dot{\theta} = \parentheses{ \frac{J_{1}}{\sigma} - 1 } v^{1}
    = \parentheses{ \frac{J_{1}}{\sigma} - 1 } \Omega_{1},
  \end{equation*}
  and substitute it to \eqref{eq:tilde_ell_r} to obtain
  \begin{equation*}
    \tilde{\ell}_{\rm r}\parentheses{ \bOmega, \bY, \parentheses{ \frac{J_{1}}{\sigma} - 1 } \Omega_{1}, \Gamma }
    = \frac{1}{2} \bOmega^{T} \tilde{\mathbb{I}} \bOmega
    + \frac{m}{2}\norm{ \bY + l \bOmega \times \bE3 }^2
    - m \mathrm{g} l \bGamma^{T} \bE3,
  \end{equation*}
  where
  \begin{equation}
    \label{eq:tildeI}
    \tilde{\mathbb{I}} \defeq \diag\parentheses{
      I_{1} + \frac{J_{1}^{2}}{\sigma} ,\, I_{2} + J_{2},\, I_{3} + J_{3}
    }.
  \end{equation}
  Notice that this controlled Lagrangian takes the same form as the Lagrangian~\eqref{eq:ell-PendulumSkate} of the original pendulum skate \textit{without} the rotor, with the only difference being the inertia tensors $\mathbb{I}$ and $\tilde{\mathbb{I}}$.
\end{example}

\section{Stabilization of Pendulum Skate}
\label{sec:stabilization}
\subsection{Stabilization of Sliding Equilibrium}
Recall from \Cref{prop:sliding_eq} that the sliding equilibrium~\eqref{eq:zeta_sl} was always unstable without control.
The control law obtained above using the controlled Lagrangian can stabilize it:
\begin{proposition}
  \label{prop:sliding_eq-controlled}
  The sliding equilibrium~\eqref{eq:zeta_sl} of the pendulum skate with an internal rotor from \Cref{fig:PendulumSkate-Controlled} is stable with the control~\eqref{eq:u-PendulumSkate} if
  \begin{equation}
    \label{eq:zeta_sl-stability_condition}
    -\frac{J_{1}^{2}}{I_{1} + m l^{2}} < \sigma < 0.
  \end{equation}
\end{proposition}
\begin{proof}
  See \Cref{proof:sliding_eq-controlled}.
\end{proof}
\begin{remark}
  In terms of $\nu \defeq -J_{1}/\sigma$, the control gain in view of \eqref{eq:u-PendulumSkate}, the above condition~\eqref{eq:zeta_sl-stability_condition} is equivalent to $\nu > (I_{1} + m l^{2})/J_{1}$, which is what we had in \cite[Eq.~(18) in Theorem~4]{GaOh2022}.
\end{remark}

\subsection{Numerical Results---Uncontrolled}
As a numerical example, consider the pendulum skate with $m = 2.00\, [\text{kg}]$, $l=0.80\, [\text{m}]$, $(I_{1}, I_{2}, I_{3}) = (0.35,0.35,0.004)\, [\text{kg}\cdot\text{m}^2]$ with $\mathrm{g} = 9.80\, [\text{m}/\text{s}^{2}]$.
As the initial condition, we consider a small perturbation to the sliding equilibrium \eqref{eq:zeta_sl}:
\begin{equation}
  \label{eq:IC}
  \begin{array}{c}
    \bOmega(0)=(0.1,\, 0.1\tan\phi_{0},\, 0.1),
    \qquad
    \bY(0)=(1,0,0),
    \qquad
    \bGamma(0)=(0, \sin\phi_{0}, \cos\phi_{0}),
  \end{array}
\end{equation}
where $\phi_{0} = 0.1$ is the small angle of tilt of the pendulum skate away from the vertical upward direction.
\Cref{fig:NumericalResults-Uncontrolled} shows the result of the simulation of the uncontrolled pendulum skate~\eqref{eq:NHEP-quasivel-PendulumSkate} with the above initial condition.
It clearly exhibits instability as the pendulum skate falls down.

\begin{figure}[htbp]
  \centering
  \captionsetup{width=0.3\textwidth}
  \subfloat[Body angular velocity $\bOmega$]{
    \includegraphics[width=.29\linewidth]{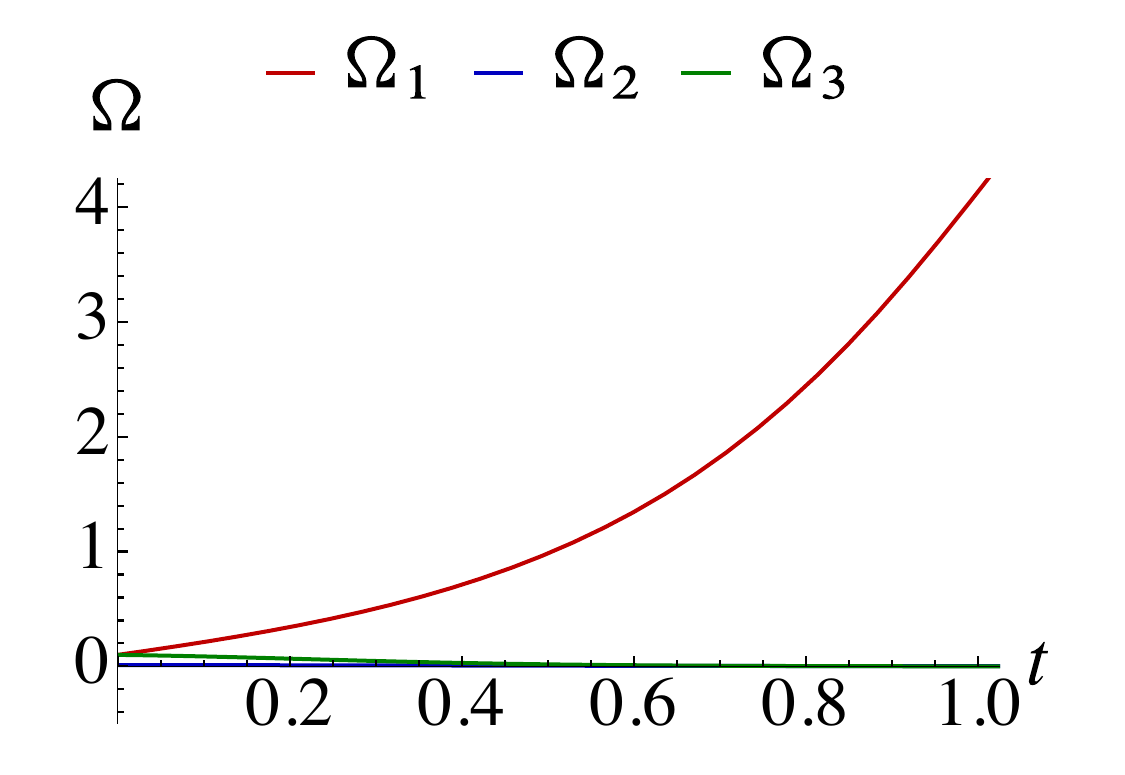}
  }
  \
  \subfloat[Velocity $Y_{1}$ in body frame]{
    \includegraphics[width=.29\linewidth]{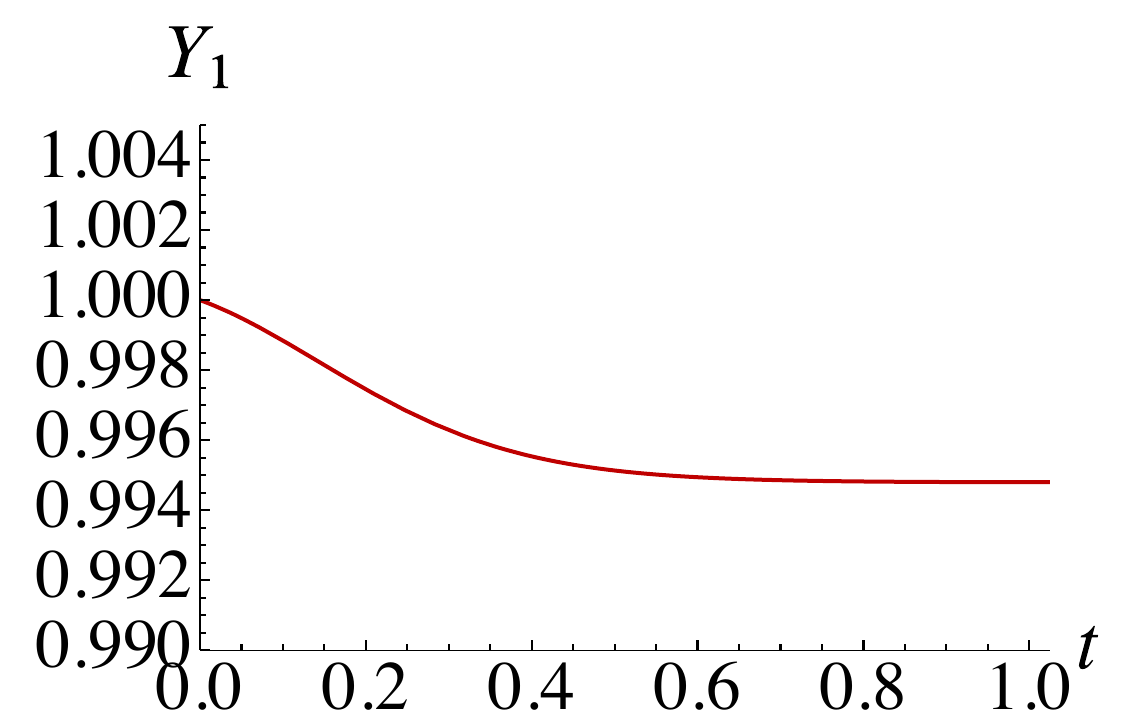}
  }
  \
  \subfloat[Vertical upward direction $\boldsymbol{\Gamma}$ seen from body frame]{
    \includegraphics[width=.29\linewidth]{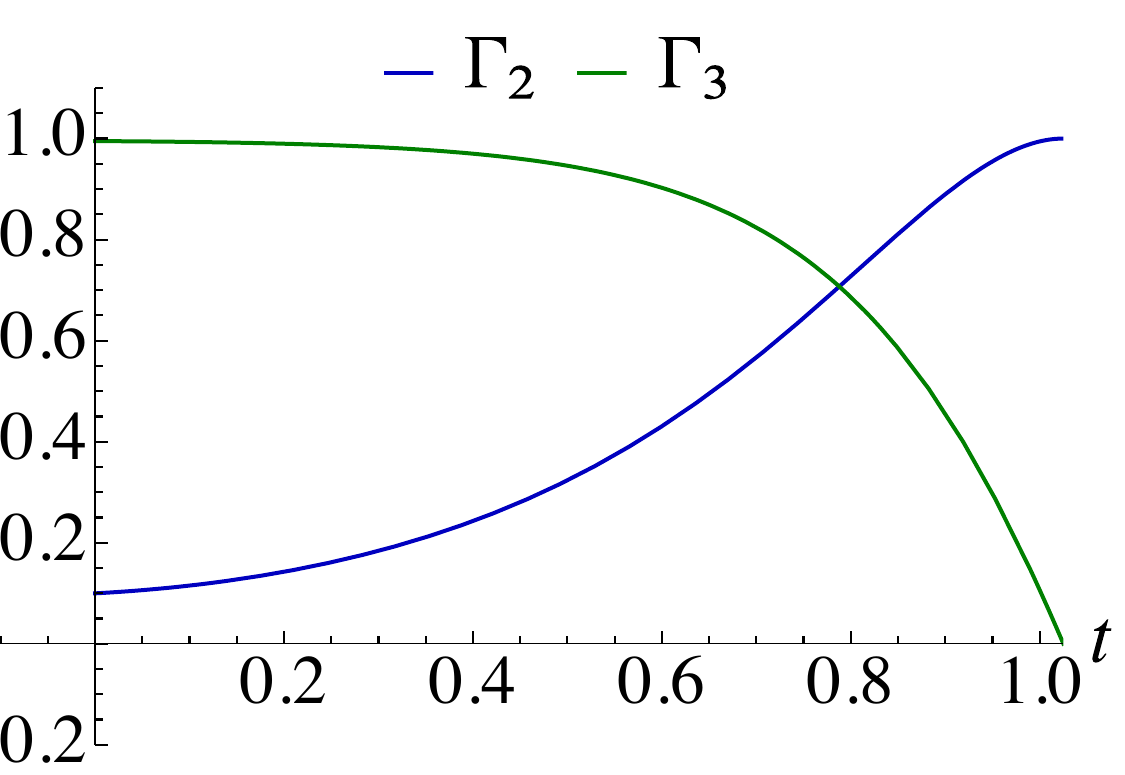}
  }
  \caption{
    Numerical results for the uncontrolled pendulum skate~\eqref{eq:NHEP-quasivel-PendulumSkate} with $m = 2.00\, [\text{kg}]$, $l=0.80\, [\text{m}]$, $(I_{1}, I_{2}, I_{3}) = (0.35,0.35,0.004)\, [\text{kg}\cdot\text{m}^2]$.
    The initial condition~\eqref{eq:IC} is a small perturbation from the sliding equilibrium~\eqref{eq:zeta_sl}.
    The pendulum skate falls down, i.e., $\Gamma_{3} = 0$ at $t \simeq 1.025$ exhibiting the instability of the sliding equilibrium shown in \Cref{prop:sliding_eq}.
  }
  \label{fig:NumericalResults-Uncontrolled}
\end{figure}

\subsection{Numerical Results---Controlled}
We also solved the controlled system~~\eqref{eq:cNHEP} (see also \Cref{ex:PendulumSkate_with_rotor}) with the control~\eqref{eq:u-PendulumSkate} using the same initial condition~\eqref{eq:IC}.
The mass of the rotor is $1\,[\text{kg}]$, making the total mass $m=3\,[\text{kg}]$; we also set $(J_{1}, J_{2}, J_{3}) = (0.005, 0.0025, 0.0025)\,[\text{kg}\cdot\text{m}^2]$.
The lower bound for $\sigma$ shown in \eqref{eq:zeta_sl-stability_condition} in order to achieve stability is $-J_{1}^{2}/(I_{1} + m l^{2}) \simeq -1.10 \times 10^{-5}$; hence we set $\sigma = -10^{-5}$ here.
\Cref{fig:NumericalResults-Controlled} shows that the system is indeed stabilized by the control.

\begin{figure}[htbp]
  \centering
  \captionsetup{width=0.3\textwidth}
  \subfloat[Body angular velocity $\bOmega$]{
    \includegraphics[width=.31\linewidth]{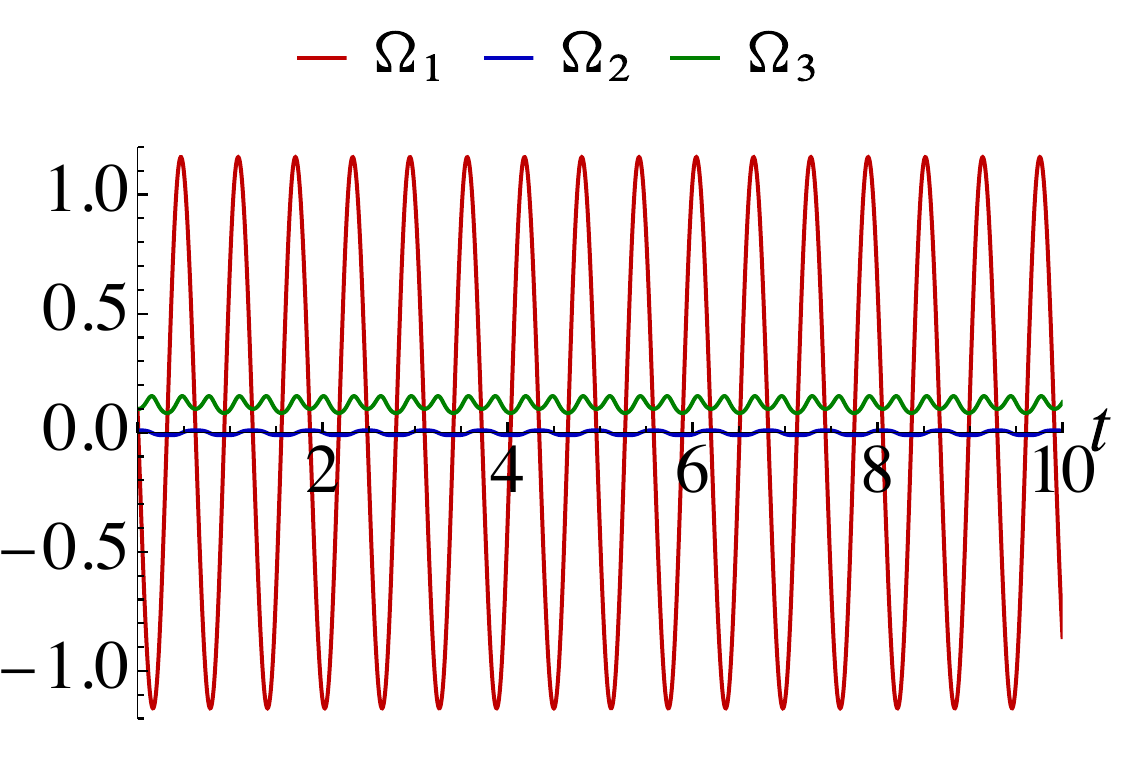}
  }
  \
  \subfloat[Velocity $Y_{1}$ in body frame]{
    \includegraphics[width=.31\linewidth]{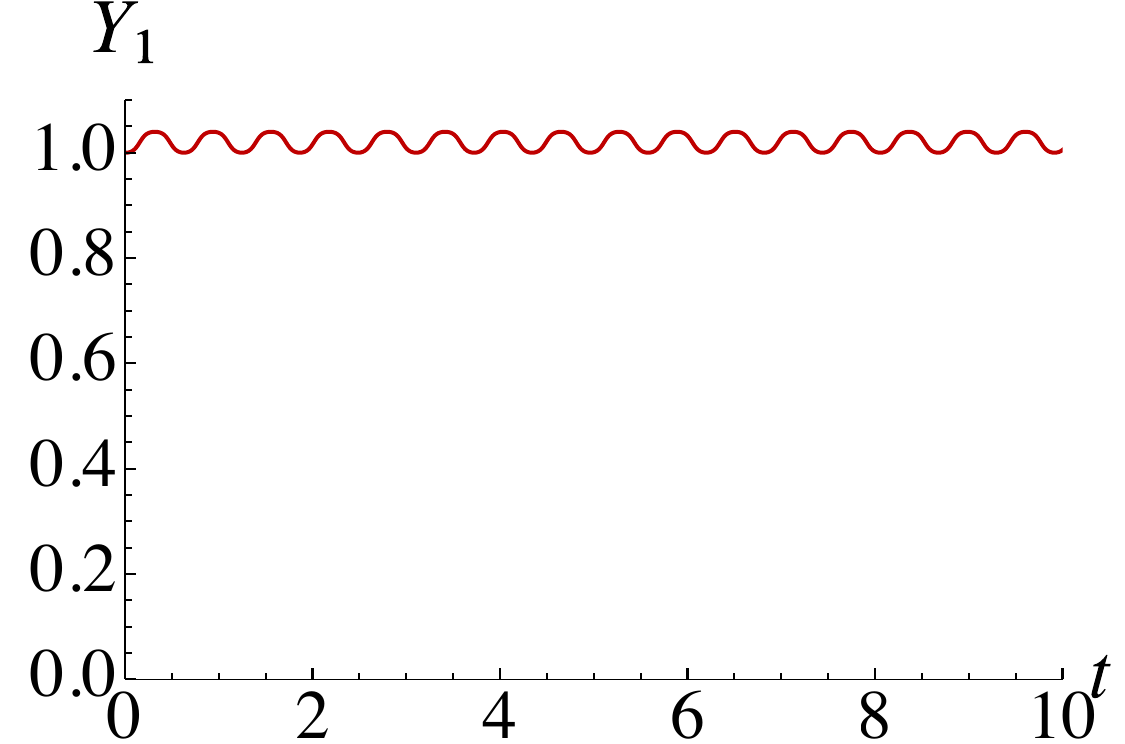}
  }
  \
  \subfloat[Vertical upward direction $\boldsymbol{\Gamma}$ seen from body frame]{
    \includegraphics[width=.31\linewidth]{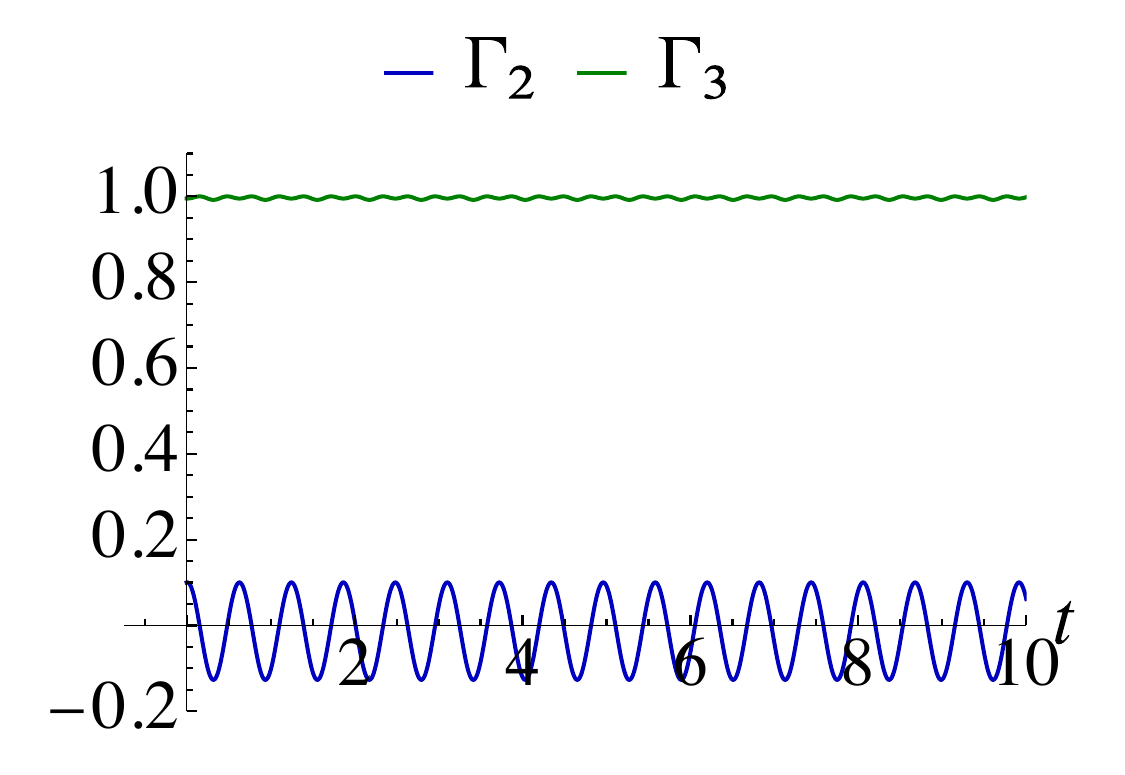}
  }
  \caption{
    Numerical results for the controlled pendulum skate~\eqref{eq:cNHEP} (see also \Cref{ex:PendulumSkate_with_rotor}) with $m = 3.00\, [\text{kg}]$, $(J_{1}, J_{2}, J_{3}) = (0.005, 0.0025, 0.0025)\,[\text{kg}\cdot\text{m}^2]$.
    The solutions are shown for the time interval $0 \le t \le 10$.
    One sees that the system is stabilized by the control in comparison to \Cref{fig:NumericalResults-Uncontrolled}; particularly, $\Gamma_{3}$ stays close to 1, indicating that it maintains its position almost upright.
  }
  \label{fig:NumericalResults-Controlled}
\end{figure}

\section*{Acknowledgments}
This paper is an extended version of our conference paper~\cite{GaOh2022}.
We would like to thank Vakhtang Putkaradze for helpful discussions and also the reviewer for the constructive criticisms and comments.
This work was supported by NSF grant CMMI-1824798.
TO also would like to thank the Graduate School of Informatics of Kyoto University and Kazuyuki Yagasaki for their hospitality in supporting his sabbatical visit at Kyoto University, where some parts of the work were performed.

\appendix

\section{Some Proofs}
\subsection{Proof of \Cref{prop:spinning_eq}}
\label{proof:spinning_eq}
The Jacobian $Df$ of the vector field $f$ from \eqref{eq:vector_field-PendulumSkate} at the spinning equilibrium~\eqref{eq:zeta_sp} is
\begin{equation*}
  Df(\zeta_{\rm sp}) =
  \begin{bmatrix}
    0 & 0 & \frac{m l \Omega _0}{I_{1} + m l^{2}} & \frac{m \mathrm{g} l - \left(I_{3} + m l^{2} - I_{2} \right) \Omega_{0}^{2} }{I_{1} + m l^{2}} & 0\\
    0 & 0 & 0 & 0 & 0 \\
    -2 l \Omega _0 & 0 & 0 & 0 & 0 \\
    1 & 0 & 0 & 0 & 0 \\
    0 & 0 & 0 & 0 & 0
  \end{bmatrix},
\end{equation*}
and its eigenvalues are
\begin{equation*}
  \left\{0, 0, 0, \pm\sqrt{ \frac{m \mathrm{g} l - (I_{3} + m l^{2}- I_{2}) \Omega_{0}^{2}}{I_{1} + m l^{2}} } \right\}.
\end{equation*}
If $I_{3} + m l^{2} < I_{2}$ then by the Instability from Linearization criterion~(see, e.g., \citet[p.216]{Sastry1999}), $\zeta_{\rm sp}$ is unstable.
On the other hand, if $I_{3} + m l^{2} > I_{2}$ then the linear analysis is inconclusive, and so we would like to use the following nonlinear method:

\begin{EC}
  \label{thm:energy-Casimir}
  Consider a system of differential equations $\dot{\zeta} = f(\zeta)$ on ${\R}^n$ with locally Lipschitz $f\colon \R^{n} \to \R^{n}$ with an equilibrium $\zeta_{0} \in \R^{n}$, i.e., $f(\zeta_{0}) = 0$.
  Assume that the system has $k + 1(< n)$ invariants $\mathscr{E}$ and $\{ C_{j} \}_{j=1}^{k}$ that are $C^{2}$ and submersive at $\zeta_{0}$, and also that the gradients $\{ DC_{j}(\zeta_{0}) \}_{j=1}^{k}$ at $\zeta_{0}$ are linearly independent.
  Then $\zeta_{0}$ is a stable equilibrium if
  \begin{enumerate}[(i)]
  \item there exist scalars $\{ c_{j} \in \R \}_{j=1}^{k}$ such that $D(\mathscr{E} + c_{1} C_{1} + \dots + c_{k} C_{k})(\zeta_{0}) = 0$; and
  \item the Hessian $\mathscr{H} \defeq D^{2}(\mathscr{E} + c_{1} C_{1} + \dots + c_{k} C_{k})(\zeta_{0})$ is sign definite on the tangent space at $\zeta_{0}$ of the submanifold $\setdef{ \zeta \in \R^{n} }{ C_{j}(\zeta) = C_{j}(\zeta_{0})\ \forall j \in \{1,\dots, k\} }$, i.e., for any $w \in{\R}^{n} \backslash\{0\}$ satisfying $DC_{j}(\zeta_{0}) \cdot w = 0$ with every $j \in \{1, \dots, k \}$, one has $w^{T} \mathscr{H} w > 0$ (or $< 0$).
  \end{enumerate}
\end{EC}
We note that, despite its name, the above theorem does not assume that the invariants are Casimirs: any invariants---Casimirs or not---would suffice. 

In order to use the above theorem, we set the constrained energy \eqref{eq:mathscrE} as $\mathscr{E}$ and use the invariants $C_{1}$ and $C_{2}$ from \eqref{eq:C1-PendulumSkate} and \eqref{eq:C2-PendulumSkate} as well as $C_{3} = (\Gamma_{2}^{2} + \Gamma_{3}^{2})/2$ with $k = 3$.
Now, since
\begin{equation*}
  D\mathscr{E}(\zeta_{\rm sp})
  = \begin{bmatrix}
    0 \\
    I_{3} \Omega_{0} \\
    0 \\
    0 \\
    I_{3} \Omega_{0}^{2} + m \mathrm{g} l
  \end{bmatrix},
  \quad
  DC_{1}(\zeta_{\rm sp})
  = \begin{bmatrix}
    0 \\
    I_{3} \\
    0 \\
    0 \\
    2 I_{3} \Omega_{0}
  \end{bmatrix},
  \quad
  DC_{2}(\zeta_{\rm sp})
  = \begin{bmatrix}
    0 \\
    0 \\
    1 \\
    2 l \Omega_{0} \\
    0
  \end{bmatrix},
  \quad
  DC_{3}(\zeta_{\rm sp})
  = \begin{bmatrix}
    0 \\
    0 \\
    0 \\
    0 \\
    1
  \end{bmatrix},
\end{equation*}
setting $(c_{1}, c_{2}, c_{3}) = (-\Omega_{0}, 0, I_{3} \Omega_{0}^{2} - m \mathrm{g} l)$, we have
\begin{equation*}
  D( \mathscr{E} + c_{1} C_{1} + c_{2} C_{2} + c_{3} C_{3} )(\zeta_{\rm sp}) = 0.
\end{equation*}
Then we also see that
\begin{align*}
  \mathscr{H}
  &\defeq D^{2}( \mathscr{E} + c_{1} C_{1} + c_{2} C_{2} + c_{3} C_{3} )(\zeta_{\rm sp}) \\
  &= \begin{bmatrix}
    I_1 + m l^{2} & 0 & 0 & 0 & 0\\
    0 & I_{3} & 0 & 0 & 0 \\
    0 & 0 & m & m l \Omega_{0} & 0 \\
    0 & 0 & m l \Omega_{0} & \Omega_{0}^{2} \left(I_{3} + m l^{2} - I_{2} \right) - m \mathrm{g} l & 0 \\
    0 & 0 & 0 & 0 & -m \mathrm{g} l\\
  \end{bmatrix}.
\end{align*}
The relevant tangent space is the null space
\begin{equation*}
  \ker
  \begin{bmatrix}
    DC_{1}(\zeta_{\rm sp})^{T} \\
    DC_{2}(\zeta_{\rm sp})^{T} \\
    DC_{3}(\zeta_{\rm sp})^{T}
  \end{bmatrix}
  = \setdef{
    w = s_{1}
    \begin{bmatrix}
      1 \\
      0 \\
      0 \\
      0 \\
      0
    \end{bmatrix}
    + s_{2}
    \begin{bmatrix}
      0 \\
      0 \\
      0 \\
      1 \\
      0
    \end{bmatrix}
  }{
    s_{1}, s_{2} \in \R
  }.
\end{equation*}
Hence we have the quadratic form
\begin{equation*}
  w^{T} \mathscr{H} w
  = (I_{1} + m l^{2}) s_{1}^{2}
  + \parentheses{
    (I_{3} + m l^{2} - I_{2}) \Omega _0^2
    - m \mathrm{g} l 
  } s_{2}^{2},
\end{equation*}
which is positive definite in $(s_{1}, s_{2})$ under the assumed conditions~\eqref{eq:zeta_sp-stability_condition}.

\subsection{Derivation of \eqref{eq:tilde_p}}
\label{proof:tilde_p}
We can obtain the expression for $\tilde{p}_{i}$ in \eqref{eq:tilde_p} as follows:
Let us first rewrite its definition as
\begin{equation*}
  \tilde{p}_{i} \defeq \ip{ \left. \fd{\tilde{\ell}_{\rm r}}{\xi} \right|_{\rm c} }{ \mathcal{E}_{i} }
  = \mathcal{E}_{i}^{j} (\Gamma)\, \ip{ \left. \fd{\tilde{\ell}_{\rm r}}{\xi} \right|_{\rm c} }{ E_{j} }.
\end{equation*}
Performing the same computations as in \cite[Eq.~(16)]{BlLeMa2000} and using $\xi^{k} = \mathcal{E}^{k}_{l}\,v^{l}$, one has
\begin{align*}
  \ip{ \fd{\tilde{\ell}_{\rm r}}{\xi} }{ E_{j} }
  &= \mathbb{G}_{jk}\, \xi^{k}
    + \mathbb{G}_{ja} \parentheses{ \dot{\theta}^{a} + \tau^{a}_{k}\, \xi^{k} }
    + \mathbb{G}_{ab}\, \rho^{bc}\, \pd{\tilde{\ell}_{\rm r}}{\dot{\theta}^{c}}\, \tau^{a}_{j} \\
  &\quad + \sigma_{ab}\, \tau^{a}_{j} \tau^{b}_{k} \, \xi^{k}
    + (\rho_{ab} - \mathbb{G}_{ab}) \rho^{ad}\, \pd{\tilde{\ell}_{\rm r}}{\dot{\theta}^{d}} \parentheses{ \mathbb{G}^{bc} \mathbb{G}_{cj} + \tau^{b}_{j} } \\
  &= \mathbb{G}_{jk}\, \mathcal{E}^{k}_{l} v^{l}
    + \mathbb{G}_{ja} \parentheses{ \dot{\theta}^{a} + \tau^{a}_{k}\, \mathcal{E}^{k}_{l}\, v^{l} }
    + \mathbb{G}_{ab}\, \rho^{bc}\, \pd{\tilde{\ell}_{\rm r}}{\dot{\theta}^{c}}\, \tau^{a}_{j} \\
  &\quad + \sigma_{ab}\, \tau^{a}_{j} \tau^{b}_{k} \, \mathcal{E}^{k}_{l}\, v^{l}
    + (\rho_{ab} - \mathbb{G}_{ab}) \rho^{ad}\, \pd{\tilde{\ell}_{\rm r}}{\dot{\theta}^{d}} \parentheses{ \mathbb{G}^{bc} \mathbb{G}_{cj} + \tau^{b}_{j} },
\end{align*}
and so
\begin{align*}
  \tilde{p}_{i}
  &= \mathcal{E}_{i}^{j} \biggl(
  \mathbb{G}_{jk}\, \mathcal{E}^{k}_{l} v^{l}
  + \mathbb{G}_{ja} \parentheses{ \dot{\theta}^{a} + \tau^{a}_{k}\, \mathcal{E}^{k}_{l}\, v^{l} }
  + \mathbb{G}_{ab}\, \rho^{bc}\, \pd{\tilde{\ell}_{\rm r}}{\dot{\theta}^{c}}\, \tau^{a}_{j} \\
  &\quad\qquad + \sigma_{ab}\, \tau^{a}_{j} \tau^{b}_{k} \, \mathcal{E}^{k}_{l}\, v^{l}
  + (\rho_{ab} - \mathbb{G}_{ab}) \rho^{ad}\, \pd{\tilde{\ell}_{\rm r}}{\dot{\theta}^{d}} \parentheses{ \mathbb{G}^{bc} \mathbb{G}_{cj} + \tau^{b}_{j} }
  \biggr) \biggr|_{\rm c}\\
  &= \mathcal{G}_{i\alpha}\,v^{\alpha}
    + \mathcal{G}_{ia} \parentheses{ \dot{\theta}^{a} + \mathcal{T}^{a}_{\alpha}\, v^{\alpha} }
    + \mathbb{G}_{ab}\,\rho^{bc}\, \tilde{\pi}_{c}\, \mathcal{T}^{a}_{i} \\
  &\quad + \sigma_{ab}\, \mathcal{T}^{a}_{i} \mathcal{T}^{b}_{\alpha} \, v^{\alpha}
    + (\rho_{ab} - \mathbb{G}_{ab}) \rho^{ad}\, \tilde{\pi}_{d} \parentheses{ \mathbb{G}^{bc} \mathcal{G}_{ci} + \mathcal{T}^{b}_{i} },
\end{align*}
where we used \eqref{eq:tilde_pi}.

\subsection{Proof of \Cref{prop:matching}}
\label{proof:matching}
One can obtain the matching conditions~\eqref{eq:matching} almost the same way as in the proof of \cite[Theorem~2.1]{BlLeMa2000}.
Indeed, the controlled Euler--Poincar\'e equations \eqref{eq:cNHEP} and the Euler--Poincar\'e equations \eqref{eq:NHEP-tildel} with the (reduced) controlled Lagrangian $\tilde{\ell}_{\rm r}$ match if $p_{i}$ in \eqref{eq:p_and_pi} and $\tilde{p}_{i}$ in \eqref{eq:tilde_p} are equal, i.e.,
\begin{equation*}
  \parentheses{ \mathcal{G}_{ib} + \sigma_{ab}\, \mathcal{T}^{a}_{i} } \mathcal{T}^{b}_{\alpha}\, v^{\alpha}
  + \mathbb{G}_{ab}\,\rho^{bc}\, \tilde{\pi}_{c}\, \mathcal{T}^{a}_{i}
  + (\rho_{ab} - \mathbb{G}_{ab}) \rho^{ad}\, \tilde{\pi}_{d} \parentheses{ \mathbb{G}^{bc} \mathcal{G}_{ci} + \mathcal{T}^{b}_{i} }
  = 0.
\end{equation*}
Let us assume $\mathcal{T}^{a}_{i} = - \sigma^{ab} \mathcal{G}_{bi}$, which is equivalent to to the first matching condition in \eqref{eq:matching} using the definitions \eqref{eq:mathcalG_ia} and \eqref{eq:mathcalT} of $\mathcal{G}_{ai}$ and $\mathcal{T}^{a}_{i}$.
Substituting $\mathcal{T}^{a}_{i} = - \sigma^{ab} \mathcal{G}_{bi}$ into the above displayed equation, we obtain
\begin{equation*}
  \mathcal{G}_{ai} = \rho_{ab}(\mathbb{G}^{bc} - \sigma^{bc}) \mathcal{G}_{ci},
\end{equation*}
but then this is satisfied if we assume the second matching condition in \eqref{eq:matching}.

For the control law, we again mimic \cite[Section~2.2]{BlLeMa2000} as follows:
Notice from \eqref{eq:p_and_pi} and \eqref{eq:tilde_pi} that
\begin{equation*}
  \pi_{a} - \mathbb{G}_{ab}\, \rho^{bd}\, \tilde{\pi}_{d}
  = -\mathbb{G}_{ab}\, \mathcal{T}^{b}_{\alpha}\, v^{\alpha}.
\end{equation*}
Using \eqref{eq:cNHEP2}, \eqref{eq:NHEP2-tildel}, and the above equality, we have
\begin{align*}
  u_{a} &= \od{}{t} \parentheses{ \pi_{a} - \mathbb{G}_{ab}\, \rho^{bd}\, \tilde{\pi}_{d} } \\
        &= -\mathbb{G}_{ab}\, \od{}{t} \parentheses{ \mathcal{T}^{b}_{\alpha}\, v^{\alpha} } \\
        &= -\mathbb{G}_{ab}\, \od{}{t} \parentheses{ \tau^{b}_{j}\, \mathcal{E}_{\alpha}^{j}\, v^{\alpha} } \\
        &= -\mathbb{G}_{ab}\, \tau^{b}_{j}\, \parentheses{
          D^{k}\mathcal{E}_{\alpha}^{j}\, \dot{\Gamma}_{k}\, v^{\alpha}
          + \mathcal{E}_{\alpha}^{j}\, \dot{v}^{\alpha}
          } \\
        &= \mathbb{G}_{ab}\, \tau^{b}_{j}\, \parentheses{
          D^{k}\mathcal{E}_{\alpha}^{j}\, \varkappa^{l}_{k\beta}\, \Gamma_{l}\, v^{\alpha} v^{\beta}
          - \mathcal{E}_{\alpha}^{j}\, \dot{v}^{\alpha}
          } \\
        &= \mathbb{G}_{ab}\, \tau^{b}_{j}\, \parentheses{
          \mathcal{E}_{i}^{j}\, \mathcal{F}^{i}_{\alpha\beta}\, v^{\alpha} v^{\beta}
          - \mathcal{E}_{\alpha}^{j}\, \dot{v}^{\alpha}
          } \\
        &= \mathbb{G}_{ab}\, \parentheses{
          \mathcal{T}^{b}_{i}\, \mathcal{F}^{i}_{\alpha\beta}\, v^{\alpha} v^{\beta}
          - \mathcal{T}^{b}_{\alpha}\, \dot{v}^{\alpha}
          },
\end{align*}
where we used \eqref{eq:NHEP2-tildel} and the definitions of $\mathcal{F}$ and $\mathcal{T}$ from \eqref{eq:mathcalF} and \eqref{eq:mathcalT} in the third last three equalities.

\subsection{Proof of \Cref{prop:sliding_eq-controlled}}
\label{proof:sliding_eq-controlled}
We employ the same Energy--Casimir method from \Cref{proof:spinning_eq}.

Recall from \Cref{ex:PendulumSkate_with_rotor} that the controlled system is equivalent to the original system~\eqref{eq:vector_field-PendulumSkate} whose inertia tensor $\mathbb{I} = \diag(I_{1}, I_{2}, I_{3})$ is replaced by $\tilde{\mathbb{I}}$ shown in \eqref{eq:tildeI}:
\begin{equation*}
  (I_{1}, I_{2}, I_{3})
  \to
  \parentheses{
    I_{1} + \frac{J_{1}^{2}}{\sigma} ,\, I_{2} + J_{2},\, I_{3} + J_{3}
  }.
\end{equation*}
Therefore, the controlled system possesses invariants $\tilde{\mathscr{E}}$ and $\{ \tilde{C}_{i} \}_{i=1}^{3}$ defined by making the above replacement in $\mathscr{E}$ and $\{ C_{i} \}_{i=1}^{3}$ (see \eqref{eq:mathscrE}, \eqref{eq:C1-PendulumSkate}, and \eqref{eq:C2-PendulumSkate}); note that $\tilde{C}_{3} = C_{3} = (\Gamma_{2}^{2} + \Gamma_{3}^{2})/2$ because it does not depend on $(I_{1}, I_{2}, I_{3})$.
More specifically, we have
\begin{equation*}
  D\tilde{\mathscr{E}}(\zeta_{\rm sl})
  = \begin{bmatrix}
    0 \\
    0 \\
    m Y_{0} \\
    0 \\
    m \mathrm{g} l
  \end{bmatrix},
  \quad
  D\tilde{C}_{1}(\zeta_{\rm sl})
  = \begin{bmatrix}
    0 \\
    I_{3} + J_{3}\\
    0 \\
    0 \\
    0
  \end{bmatrix},
  \quad
  D\tilde{C}_{2}(\zeta_{\rm sl})
  = \begin{bmatrix}
    0 \\
    0 \\
    1 \\
    0 \\
    0
  \end{bmatrix},
  \quad
  D\tilde{C}_{3}(\zeta_{\rm sl})
  = \begin{bmatrix}
    0 \\
    0 \\
    0 \\
    0 \\
    1
  \end{bmatrix}.
\end{equation*}
Setting $(c_{1}, c_{2}, c_{3}) = (0, -m Y_{0}, -m \mathrm{g} l)$, we have
\begin{equation*}
  D\parentheses{ \tilde{\mathscr{E}} + c_{1} \tilde{C}_{1} + c_{2} \tilde{C}_{2} + c_{3} \tilde{C}_{3} }(\zeta_{\rm sl}) = 0.
\end{equation*}
Then we also see that
\begin{align*}
  \mathscr{H}
  &\defeq D^{2}\parentheses{ \tilde{\mathscr{E}} + c_{1} \tilde{C}_{1} + c_{2} \tilde{C}_{2} + c_{3} \tilde{C}_{3} }(\zeta_{\rm sl}) \\
  &= \begin{bmatrix}
    I_1 + m l^{2} + J_{1}^{2}/\sigma & 0 & 0 & 0 & 0\\
    0 & I_{3} + J_{3} & 0 & -m l Y_{0} & 0 \\
    0 & 0 & m & 0 & 0 \\
    0 & -m l Y_{0} & 0 & -m \mathrm{g} l & 0 \\
    0 & 0 & 0 & 0 & -m \mathrm{g} l\\
  \end{bmatrix}.
\end{align*}
The relevant tangent space is the null space
\begin{equation*}
  \ker
  \begin{bmatrix}
    D\tilde{C}_{1}(\zeta_{\rm sl})^{T} \\
    D\tilde{C}_{2}(\zeta_{\rm sl})^{T} \\
    D\tilde{C}_{3}(\zeta_{\rm sl})^{T}
  \end{bmatrix}
  = \setdef{
    w = s_{1}
    \begin{bmatrix}
      1 \\
      0 \\
      0 \\
      0 \\
      0
    \end{bmatrix}
    + s_{2}
    \begin{bmatrix}
      0 \\
      0 \\
      0 \\
      1 \\
      0
    \end{bmatrix}
  }{
    s_{1}, s_{2} \in \R
  }.
\end{equation*}
Hence we have the quadratic form
\begin{equation*}
  w^{T} \mathscr{H} w
  = \parentheses{ I_{1} + m l^{2} + \frac{J_{1}^{2}}{\sigma} } s_{1}^{2}
  - m \mathrm{g} l\,s_{2}^{2},
\end{equation*}
which is negative definite in $(s_{1}, s_{2})$ under the assumed condition~\eqref{eq:zeta_sl-stability_condition}.

\bibliography{CtrlNHEPwithBsym}
\bibliographystyle{plainnat}

\end{document}